\documentclass[10pt]{article}
\usepackage{bbm}
\usepackage{amsfonts}
\usepackage{amsmath}
\allowdisplaybreaks
\usepackage{amssymb}
\usepackage{fancyhdr}
\usepackage{setspace}
\usepackage{latexsym}
\usepackage{mathrsfs}
\usepackage{amsthm}

\usepackage{cite,enumitem,graphicx}
\usepackage[backref,colorlinks,linkcolor=red,anchorcolor=green,citecolor=blue]{hyperref}
\usepackage{cleveref}
\usepackage[english]{babel}

\numberwithin{equation}{section}
\newtheorem{theorem}{Theorem}[section]
\newtheorem{proposition}[theorem]{Proposition}
\newtheorem{lemma}[theorem]{Lemma}
\newtheorem{corollary}[theorem]{Corollary}
\newtheorem{definition}[theorem]{Definition}
\newtheorem{remark}[theorem]{Remark}
\newtheorem{example}[theorem]{Example}

\newcommand{\R}{\mathbb R}

\newcommand{\bt}{\begin{theorem}}
\newcommand{\et}{\end{theorem}}
\newcommand{\bl}{\begin{lemma}}
\newcommand{\el}{\end{lemma}}
\newcommand{\bd}{\begin{definition}}
\newcommand{\ed}{\end{definition}}
\newcommand{\bc}{\begin{corollary}}
\newcommand{\ec}{\end{corollary}}
\newcommand{\bp}{\begin{proof}}
\newcommand{\ep}{\end{proof}}
\newcommand{\bx}{\begin{example}}
\newcommand{\ex}{\end{example}}
\newcommand{\bi}{\begin{exercise}}
\newcommand{\ei}{\end{exercise}}
\newcommand{\bo}{\begin{proposition}}
\newcommand{\eo}{\end{proposition}}
\newcommand{\br}{\begin{remark}}
\newcommand{\er}{\end{remark}}
\newcommand{\be}{\begin{equation}}
\newcommand{\ee}{\end{equation}}
\newcommand{\ba}{\begin{align}}
\newcommand{\ea}{\end{align}}
\newcommand{\bn}{\begin{enumerate}}
\newcommand{\en}{\end{enumerate}}
\newcommand{\bg}{\begin{align*}}
\newcommand{\bcs}{\begin{cases}}
\newcommand{\ecs}{\end{cases}}

\newcommand{\bean}{\begin{eqnarray*}}
\newcommand{\eean}{\end{eqnarray*}}

\def \N{{\mathbb N}}
\def \O{\Omega}

\makeatletter 
\renewcommand\theequation{\thesection.\arabic{equation}}
\@addtoreset{equation}{section}
\makeatother
\numberwithin{equation}{section}
\begin{document}

\begin{center}
\textbf{Variational method for fractional Hamiltonian system in bounded domain}\\
\end{center}

\begin{center}
Weimin Zhang\\
School of Mathematical Sciences,  Key Laboratory of Mathematics and Engineering Applications (Ministry of Education) \& Shanghai Key Laboratory of PMMP,  East China Normal University, Shanghai 200241, China
\end{center}
\begin{center}
\renewcommand{\theequation}{\arabic{section}.\arabic{equation}}
\numberwithin{equation}{section}
\footnote[0]{\hspace*{-7.4mm}
{E-mail addresses: zhangweimin2021@gmail.com (W. Zhang).}}
\end{center}

\begin{abstract}
Here we consider the following fractional Hamiltonian system
\begin{equation*}
\begin{cases}
\begin{aligned}
(-\Delta)^{s} u&=H_v(u,v) \;\;&&\text{in}~\Omega,\\
(-\Delta)^{s} v&=H_u(u,v) &&\text{in}~\Omega,\\
u &= v = 0 &&\text{in} ~ \mathbb{R}^N\setminus\Omega,
\end{aligned}
\end{cases}
\end{equation*}
where $s\in (0,1)$, $N>2s$, $H \in C^1(\mathbb{R}^2, \mathbb{R})$ and $\Omega \subset \mathbb{R}^N$ is a smooth bounded domain. 

To apply the variational method for this problem, the key question is to find a suitable functional setting. Instead of usual fractional Sobolev spaces, we use the solutions space of $(-\Delta)^{s}u=f\in L^r(\Omega)$ for $r\ge 1$, for which we show the (compact) embedding properties. When $H$ has subcritical and superlinear growth, we construct two frameworks, respectively with interpolation space method and dual method, to show the existence of nontrivial solution. 
As byproduct, we revisit the fractional Lane-Emden system, i.e. $H(u, v)=\frac{1}{p+1}|u|^{p+1}+\frac{1}{q+1}|v|^{q+1}$, and consider the existence, uniqueness of (radial) positive solutions under subcritical assumption.
\end{abstract}

\section{Introduction and main results}
In present paper, we are interested in the following fractional system of Hamiltonian type:
\begin{equation}\label{main}
\begin{cases}
\begin{aligned}
(-\Delta)^{s} u&=H_v(u,v) \;\;&&\text{in}~\Omega,\\
(-\Delta)^{s} v&=H_u(u,v) &&\text{in}~\Omega,\\
u &= v = 0 &&\text{in} ~ \mathbb{R}^N\setminus\Omega,
\end{aligned}
\end{cases}
\end{equation}
where $s\in (0,1)$, $N>2s$, $H \in C^1(\mathbb{R}^2, \mathbb{R})$ and $\Omega \subset \mathbb{R}^N$ is a bounded domain. We will not focus on the regularity condition of the domain $\Omega$, and say simply that $\Omega$ is smooth, even most results work with $C^{1, 1}$ boundary.
The study of system \eqref{main} is mainly motivated by the following classical Hamiltonian system:
\begin{equation}\label{2211101840}
\begin{cases}
\begin{aligned}
-\Delta u&=H_v(u,v) \,&&\text{in}~\Omega,\\
-\Delta v&=H_u(u,v) &&\text{in}~\Omega,\\
u &= v = 0 &&\text{on} ~ \partial\Omega.\\
\end{aligned}
\end{cases}
\end{equation}
Formally the energy functional of \eqref{2211101840} is
\begin{equation}\label{2402041528}
\mathcal{L}(u, v)=\int_{\Omega}\nabla u \nabla v dx-\int_{\Omega}H(u, v) dx.
\end{equation}
A crucial question is to decide on which space we will consider the functional $\mathcal{L}$. Let us look at the famous Lane-Emden system where $H(u,v) = \frac{|u|^{p+1}}{p+1} + \frac{|v|^{q+1}}{q+1}$, $p, q > 0$, i.e.
\begin{equation}\label{2302241613}
\begin{cases}
\begin{aligned}
-\Delta u& =|v|^{q-1}v \,&&\text{in}~\Omega,\\
-\Delta v& =|u|^{p-1}u \,&&\text{in}~\Omega,\\
u &= v = 0~~&&\text{on} ~\partial\Omega.\\
\end{aligned}
\end{cases}
\end{equation}
It is well known that the existence of positive solutions to \eqref{2302241613} on a ball is decided by the position of $(p, q)$ with respect to the critical hyperbola
\begin{equation}\label{2212031636}
p, q > 0, \quad \frac{1}{p+1}+\frac{1}{q+1}=\frac{N-2}{N}.
\end{equation}
A naive choice of the functional space for $\mathcal{L}$ is $H_0^1(\Omega)^2$. However, the Sobolev embedding greatly restricts the growth of $H$, and we could only work with $\max(p, q) \leq \frac{N+2}{N-2}$, hence many other choices of $H$ were eliminated. Another difficulty is to deal with the strong indefiniteness of the quadratic part in $\mathcal{L}$, which is neither bounded from below nor from above on any subspace of $H_0^1(\Omega)^2$ with finite codimension.

To include more choices of $H$ as for \eqref{2302241613}, people thought about functional spaces without symmetry in $u$ and $v$, such as $W_0^{1,t}(\Omega)\times W_0^{1,\frac{t}{t-1}}(\Omega)$ with $t>1$. But a new difficulty occurs since this choice is not a Hilbert space when $t\neq 2$, which prevents us from applying the linking theory due to Benci and Rabinowitz \cite{BR1979}.

As far as we are aware, Hulshof and van der Vorst \cite{Hv1993} first used the interpolation space framework to deal with the system with $H(u, v) = F(u)+G(v)$, that is
\begin{equation}\label{2402041616}
\begin{cases}
\begin{aligned}
-\Delta u & =g(v) \;\;&&\text{in}~\Omega,\\
-\Delta v& =f(u) &&\text{in}~\Omega,\\
u&= v=0 &&\text{on} ~ \partial\Omega.
\end{aligned}
\end{cases}
\end{equation}
They replaced the first integral in \eqref{2402041528} by
\[
\int_{\Omega}\widetilde{A}^{t}u\widetilde{A}^{2-t}v dx \quad \text{with}\;\; t\in (0, 2),
\]
where $\widetilde{A}^t: \widetilde{H}^{t}(\Omega)\to L^2(\Omega)$ is a family of interpolation operators and $\widetilde{H}^{t}(\Omega)$ is a family of interpolation spaces between $L^2(\Omega)$ and $H^2(\Omega)\cap H_0^1(\Omega)$. More precisely, for any
\begin{equation}\label{2211201519}
p, q > 0, \quad \frac{1}{p+1}+\frac{1}{q+1}>\frac{N-2}{N}, \end{equation}
there exist some $t\in (0, 2)$ such that the embedding
$$\widetilde{H}^t(\Omega)\times \widetilde{H}^{2-t}(\Omega)\subset L^{p+1}(\Omega)\times L^{q+1}(\Omega)$$
 is compact, which yields the existence of nontrivial solution to \eqref{2402041616} under following conditions
\begin{itemize}
\item (Subcritical) $f(u)=O(|u|^p), g(v)=O(|v|^q)$ as $|u|, |v|\to \infty$ with $p, q>1$ satisfying \eqref{2211201519};
\item (Superlinear) $f(u)=o(|u|)$, $g(v)=o(|v|)$ as $|u|, |v|\to 0$;
\item ($AR$ condition) $f, g\in C(\mathbb{R})$, $f(0)=g(0)=0$, and there exist $\gamma>2$, $R >0$ such that
\[
0<\gamma F(u)\le uf(u),\quad 0<\gamma G(v)\le vg(v),\quad \forall\; |u|, |v|\ge R.
\]
where $F(t):=\displaystyle \int_0^{t}f(\tau) d\tau\ge 0$, $G(t):=\displaystyle \int_0^{t}g(\tau) d\tau\ge 0$.
\end{itemize}
This generalized clearly the study of Lane-Emden system  \eqref{2302241613}.

Later, de Figueiredo and Felmer also applied  in \cite{dF1994} the interpolation space method to handle \eqref{2211101840}, with more general coupled $H$ where $p, q$ satisfy \eqref{2211201519},
\[
0 \le H(u, v)\le C(|u|^{p+1}+|v|^{q+1}),
\]
and $H$ is superlinear at 0, that is
\[
C(|u|^{p+1}+|v|^{q+1}) \le H(u, v) \;\;\text{with}\;\; pq>1, |u|+|v|< r \mbox{ for some } r > 0.
\]
In addition, Ambrosetti-Rabinowitz ($(AR)$ for short) type condition (as $(H2)$ below) was assumed. They proved then \eqref{2211101840} admits a nontrivial solution if
\[
(N-4)\max\{p, q\}<N+4.
\]
Remark that to work with more general coupled $H$, we need to restrict the upper bound of $p, q$.

Cl\'ement and van der Vorst in \cite{CV1995} proposed another approach to study  \eqref{2211101840}, their idea is to took advantage of dual method developed by Clarke, Ekeland and Temam \cite{CE_CPAM1980, ET1999}. Here the nonlinearity $H$ is assumed strictly convex, subcritical (see \eqref{2211201519}) and superlinear at 0. This dual method will transform \eqref{2211101840} into a problem where the energy functional possesses a mountain pass structure. In fact, consider
the Legendre-Fenchel transform of $H$ (see \cite{ET1999}, \cite[Chapter I, Section 6]{Struwe2008}):
\[
 H^*(x)=\underset{w\in \mathbb{R}^2}{\sup}\big\{\langle w, x\rangle -H(x)\big\}\quad \mbox{ for any } x \in \mathbb{R}^2,
\]
one can obtain solutions to \eqref{2211101840} by critical points of
\begin{equation}\label{2211011351}
\mathcal{J}(u, v)=\int_{\Omega}{H}^*(u, v) dx-\int_{\Omega}v\mathcal{A}u dx,\quad \forall \, (u, v)\in L^{1+\frac1p}(\Omega)\times L^{1+\frac1q}(\Omega)
\end{equation}
where $\mathcal{A}$ is the inverse of
\begin{align}
\label{isop}
-\Delta : W^{2, 1+\frac1p}(\Omega)\cap W_0^{1, 1+\frac1p}(\Omega)\to L^{1+\frac1p}(\Omega).
\end{align}
If moreover $H$ satisfies $(AR)$ type condition (similar to $(H5)$ and $(H7)$ below), the existence of nontrivial solution to \eqref{2211101840} was shown in \cite{CV1995}. This dual method was also used to handle system with critical growth, see Hulshof, Mitidieri and van der Vorst \cite{HMv1998}.

Some other approaches exist. de Figueiredo, do \'{O} and Ruf \cite{ddR_JFA2005} used Orlicz-space to obtain nontrivial solutions of \eqref{2402041616}, they replaced $W_0^{1,t}(\Omega)\times W_0^{1,\frac{t}{t-1}}(\Omega)$ by Sobolev-Orlicz space $W_0^1L_A(\Omega)\times W_0^1L_{\widetilde A}(\Omega)$, where $A$ is a $N$-function and ${\widetilde A}$ is its Young-conjugate. Owing to the fact that this setting is not a Hilbert space, they used finite-dimensional approximation method. Their models contain also nonlinearites with nearly critical growth.
The Lyapunov-Schmidt reduction approach was also applied to problem \eqref{2402041616}, see for instance \cite{RTZ2009, RT2008}. For more literature in this topic, we refer to \cite{BdT2014} and references therein.

Coming back to the special case \eqref{2302241613}. In \cite{CdM2015}, the existence of positive solution to \eqref{2302241613} was firstly considered based on topological method. When $p, q>0$ and $pq<1$,  the uniqueness of positive solution was investigated in \cite{Dalmasso2000}. In \cite{BMR2012}, \eqref{2302241613} was reduced to the following single equation
\begin{equation}\label{2302241754}
\begin{cases}
\begin{aligned}
\Delta \big(|\Delta u|^{\frac1p -1}\Delta u\big) & =|u|^{q-1}u~~&& \text{in}~\Omega,\\
u & = \Delta u=0 &&\text{on}~\partial\Omega.
\end{aligned}
\end{cases}
\end{equation}
If $p, q>0$, $pq\neq 1$ and subcritical as in \eqref{2211201519}, then \eqref{2302241613} admits a positive classical ground state solution (it has minimal energy among all solutions), see \cite{BMR2012}. If $\Omega$ is a ball, we can use the Schwartz rearrangement to show that the ground state solution is radially symmetric. Furthermore, \eqref{2302241613} has no positive solutions, if $\Omega$ is star-shaped and $p, q > 0$ do not satisfy \eqref{2211201519}, see \cite{Mitidieri1993} .

\medskip
To our best knowledge, for the fractional Laplacian case, although some special cases such as fractional Lane-Emden systems were studied (see \cite{LM2019} and references therein), there exists no study of \eqref{main} for general coupled subcritical nonlinearities so far. As mentioned above, a key step to handle \eqref{main} with the variational approach is to establish a suitable functional framework. Furthermore, in the Laplacian case, regardless of interpolation method, dual method, or reduced into a single equation, one needs the isomorphism given by \eqref{isop}. In the fractional Laplacian case, we need to find suitable functional space which plays the role of $W^{2, 1+\frac1p}(\Omega)\cap W_0^{1, 1+\frac1p}(\Omega)$.

\subsection{Weak solution and fractional spaces}
Let $\Omega$ be a smooth bounded domain and $s \in (0, 1)$, we denote
\[
X_0^s(\Omega):=\left\{u\in L^2(\mathbb{R}^N):u=0~ \mbox{ a.e.~in}~\mathbb{R}^N \setminus \Omega \;\;\text{and}\; \;\int_{\mathbb{R}^{2N}}{\frac{|u(x)-u(y)|^2}{|x-y|^{N+2s}}}dxdy<\infty\right\},
\]
endowed with the norm
\[
\|u\|_{X_0^s(\Omega)}:=\left( \int_{\mathbb{R}^{2N}}{\frac{|u(x)-u(y)|^2}{|x-y|^{N+2s}}}dxdy\right)^{1/2}.
\]
The embedding $X_0^s(\Omega)\hookrightarrow L^r(\Omega)$ is continuous for $r\in [1,\frac{2N}{N-2s}]$ and compact for $r\in [1,\frac{2N}{N-2s})$, see \cite[Theorems 6.5, 7.1]{Di12}.
$(-\Delta)^s$ is an operator from $X_0^s(\Omega)$ into its dual space, namely
\[
\langle (-\Delta)^s u, v \rangle=\int_{\mathbb{R}^{2N}}{\frac{(u(x)-u(y))(v(x)-v(y))}{|x-y|^{N+2s}}}dxdy,\quad \forall\, u, v\in X_0^s(\Omega).
\]
Formally the energy functional associated to \eqref{main} is
\begin{equation}\label{2401111951}
\mathcal{K}(u, v)=\int_{\mathbb{R}^{2N}}{\frac{(u(x)-u(y))(v(x)-v(y))}{|x-y|^{N+2s}}}dxdy-\int_{\Omega}H(u, v) dx.
\end{equation}
As mentioned before, the crucial question is to find suitable functional space to work with $\mathcal{K}$, and $X_0^s(\Omega)^2$ would not be the right one if we hope to handle more general functional $H$.

Next we recall the regularity result for the linear equation
\begin{equation}\label{2306011244}
(-\Delta)^{s} u\;=f \;\;\text{in}~\Omega,\quad
u =0~\;\;\text{in} ~ \mathbb{R}^N\setminus\Omega.
\end{equation}
We denote
\[
\delta(x):=\text{dist}(x, \partial\Omega),\quad x\in\Omega,
\]
and
\[
C_{\delta}^{\alpha}(\overline{\Omega}):= \Big\{u\in C(\overline\Omega): \frac{u}{\delta^s} \mbox{ admits a continuous extension belonging to } C^{\alpha}(\overline{\Omega})\Big\}.
\]
Ros-Oton and Serra in \cite[Proposition 1.1, Theorem 1.2]{RS2014} proved that
\begin{lemma}\label{2402061816}
Let $\Omega$ be a bounded $C^{1,1}$ domain and $f\in L^{\infty}(\Omega)$. There exists a unique $u\in X_0^s(\Omega)$ solving \eqref{2306011244}.  Moreover $u\in C^s(\overline{\Omega})\cap C_{\delta}^{\alpha}(\overline{\Omega})$, and
\[
\|u\|_{C^s(\overline{\Omega})}+ \left\|\frac{u}{\delta^s}\right\|_{C^\alpha(\overline{\Omega})}\le C\|f\|_{\infty}
\]
for some $0<\alpha<\min\{s, 1-s\}$. The constants $\alpha$, $C$ depend only on $\Omega$ and $s$.
\end{lemma}

Here we shall consider weaker solution (see definition below) and choose the test function space as
\[
\mathcal{T}_s(\Omega):=\{u\in X_0^s(\Omega): (-\Delta)^s u\in  C_c^{\infty}(\Omega)\}.
\]
Hence $\mathcal{T}_s(\Omega) \subset C^s(\overline{\Omega})\cap C_{\delta}^{\alpha}(\overline{\Omega})$. Let
\[
L^1(\Omega; \delta^s dx) :=\Big\{u\in L^1_{loc}(\Omega): \int_{\Omega} |u|\delta^sdx<\infty\Big\}
\]
be endowed with the norm $\|u\|_{L^1(\Omega; \delta^s dx)}= |u \delta^s|_1$. In this paper, $|\cdot|_r$ denotes always the norm of $L^r(\Omega)$.
\begin{definition}\label{2211161306}
Let $s\in (0, 1)$, $N>2s$ and $f\in L^1(\Omega; \delta^s dx)$. We say that $u$ is a $L^1$-weak solution to \eqref{2306011244} if $u\in L^1(\Omega)$, and
\begin{align}
\label{L1weak}
\int_{\Omega}u(-\Delta)^s\varphi dx=\int_{\Omega}f \varphi dx,\quad  \forall \, \varphi\in \mathcal{T}_s(\Omega).
\end{align}
\end{definition}
\noindent
Note that similar notions were given in \cite{RS2014-1, LPPS2015}. We will prove in Proposition \ref{2304232341} that for any $f\in L^1(\Omega; \delta^s dx)$, there exists a unique $L^1$-weak solution to \eqref{2306011244}. We give also the comparison principle and maximum principle for fractional Laplacian in $L^1$-weak sense, see Lemmas \ref{2306081933} and \ref{2306012122}. Accordingly, we define
\begin{definition}
For $s\in (0, 1)$ and $N>2s$, $(u, v)$ is said a $L^1$-weak solution to system \eqref{main} if $u, v \in L^1(\Omega)$, $H_u(u, v), H_v(u, v)\in L^1(\Omega; \delta^sdx)$ and for any $(\varphi, \psi)\in \mathcal{T}_s(\Omega)\times \mathcal{T}_s(\Omega)$,
\[
\int_{\Omega} u(-\Delta)^s\varphi dx=\int_{\Omega} H_v(u, v)\varphi dx,\quad
\int_{\Omega} v(-\Delta)^s\psi dx=\int_{\Omega} H_u(u, v)\psi dx.
\]
If in addition $u, v\in L^{\infty}(\Omega)$, we call $(u, v)$ a {\sl classical solution}.
\end{definition}

Another important functional space for us is the set of $u$ such that $(-\Delta)^s u\in L^r(\Omega)$, namely for $r \geq 1$,
\begin{equation}\label{2401141656}
\mathcal{W}^{2s, r}(\Omega):=\{u\in L^1(\Omega): \exists\; f \in L^r(\Omega) \mbox{ such that \eqref{L1weak} holds true}\}
\end{equation}
endowed with the norm
$$\|u\|_{\mathcal{W}^{2s, r}(\Omega)}=|(-\Delta)^s u|_r.$$
We shall use $\mathcal{W}^{2s, r}(\Omega)$ to play the role of $W^{2, r}(\Omega)\cap W_0^{1, r}(\Omega)$ in the Laplacian case. In particular, we denote
\[
\mathcal{H}^{2s}(\Omega)=\mathcal{W}^{2s, 2}(\Omega).
\]
Some (compact) embedding properties of $\mathcal{W}^{2s, r}(\Omega)$ will be shown in Proposition \ref{2302102144}. 

\subsection{Fractional Hamiltonian system}
Motivated by \cite{Hv1993, dF1994}, we apply firstly interpolation space method to study the system \eqref{main}. For $0\le\alpha\le 2s$, consider the interpolation space
\begin{equation}\label{2406191546}
E^\alpha:=\Big\{u=\sum_{j\ge 1}a_j\varphi_j\in L^2(\Omega): \sum_{j\ge 1}\lambda_j^{\frac{\alpha}{s}}a_j^2<\infty \Big\},
\end{equation}
where $\lambda_j$ is the $j$-th eigenvalue of $(-\Delta)^s$ with corresponding eigenfunction $\varphi_j$, and $\{\varphi_j\}$ forms an orthonormal basis of $L^2(\Omega)$. For $\alpha \in (0, 2s)$, let $A^{\alpha}: E^{\alpha}\to L^{2}(\Omega)$ be given by
\begin{equation}\label{2406191601}
A^{\alpha}u:=\sum_{j\ge 1}\lambda_j^{\frac{\alpha}{2s}}a_j\varphi_j \in L^2(\Omega),\quad \forall \;u=\sum_{j\ge 1}a_j\varphi_j\in E^\alpha.
\end{equation}
$E^\alpha$ is clearly a Hilbert space with the scalar product
\[
(u,v)_{E^\alpha}= \int_{\Omega}A^{\alpha}uA^{\alpha}v dx = \sum_{j\ge 1}\lambda_j^{\frac{\alpha}{s}}\langle u, \varphi_j\rangle_{L^2}\langle v, \varphi_j\rangle_{L^2}.
\]
To handle the Hamiltonian system \eqref{main}, we define
\[
\mathbf{E}_{\alpha}:=E^{\alpha}\times E^{2s-\alpha}, \quad \mbox{$\alpha \in (0, 2s)$}.
\]
Instead of considering $\mathcal K$ over $X_0^s(\Omega)^2$, we consider the energy functional
\begin{equation}\label{2306171702}
\mathcal{E}(u,v)=\int_{\Omega}A^{\alpha}uA^{2s-\alpha}v dx-\int_{\Omega}H(u, v) dx,\quad \forall\; (u,v)\in \mathbf{E}_{\alpha}.
\end{equation}
The choice of $\mathbf{E}_{\alpha}$ originates from three observations.
\begin{itemize}
\item $\mathcal{H}^{2s}(\Omega) = E^{2s}$, seeing Proposition \ref{2209072124};
\item $E^\alpha$ can be embedded compactly into some $L^r(\Omega)$, seeing Proposition \ref{2302102144} and Remark \ref{2211231543};
\item Every critical point of $\mathcal{E}$ is a $L^1$-weak solution to \eqref{main}, seeing Proposition \ref{2302302237}.
\end{itemize}
Applying a linking theorem in \cite{Felmer1993}, we show the following existence result, which extends the study for the Laplacian case in \cite{dF1994}.
\begin{theorem}\label{220516}
Let $N>2s$ and $p, q>0$ satisfy \begin{equation}\label{condition}
 1>\frac1{p+1}+\frac1{q+1}>\frac{N-2s}{N},
\end{equation}
 and
\begin{equation}\label{Condition_pq}
(N-4s)\max\{p, q\}<N+4s.
\end{equation}
Assume that $H \in C^1(\mathbb{R}^2, \mathbb{R})$ satisfies
\begin{itemize}
\item[$(H1)$] $H \ge 0$ in $\mathbb{R}^2$;
\item[$(H2)$] There exists $R>0$ such that
\begin{equation}\label{2212032128}
\frac{1}{p+1}H_u(u,v)u+\frac{1}{q+1}H_v(u,v)v\ge H(u,v)>0, \quad \forall\; |u|+|v|\ge R;
\end{equation}
\item[$(H3)$] There exist $r>0$, $C>0$ such that
\begin{equation}\label{2212051248}
|H(u,v)|\le C(|u|^{p+1}+|v|^{q+1}), \quad \forall \; |u|+|v|\le r;
\end{equation}
\item[$(H4)$] There exists $C>0$ such that for any $(u, v) \in \R^2$,
\begin{equation}\label{2211201618}
|H_u(u,v)|\le C\Big(|u|^p+|v|^\frac{p(q+1)}{p+1}+1\Big),\quad |H_v(u,v)|\le C\Big(|v|^q+|u|^\frac{q(p+1)}{q+1}+1\Big).
\end{equation}
\end{itemize}
Then, there exists a nontrivial classical solution to \eqref{main}.
\end{theorem}


Note that Theorem \ref{220516} requires the assumption \eqref{Condition_pq}. For getting the existence of solutions to \eqref{main} in a more broad range of $p$, $q$, we will apply also the dual method. Three major difficulties subsist.
\begin{itemize}
\item The first problem is still to construct suitable functional framework. We choose (see section \ref{2306141749}) the energy functional as follows
\begin{equation}\label{2211011351}
\mathcal{J}(f, g)=\mathcal{H}^*(f, g)-\int_{\Omega}g\mathcal{A}f dx,\quad \forall \, (f, g)\in X:=L^{1+\frac1p}(\Omega)\times L^{1+\frac1q}(\Omega)
\end{equation}
where $\mathcal{A}$ is the inverse of $(-\Delta)^s: \mathcal{W}^{2s, 1+\frac1p}(\Omega)\to L^{1+\frac1p}(\Omega)$ and $\mathcal{H}^*$ denotes the Legendre-Fenchel transform of 
\[
\mathcal{H}(u, v)=\int_{\Omega}H(u, v) dx, \quad \forall\; (u, v)\in X^*=L^{p+1}(\Omega)\times L^{q+1}(\Omega).
\]
\item Secondly, we need to verify the differentiability of $\mathcal{J}$, and the correspondence between critical points of $\mathcal{J}$ and solutions to \eqref{main}. It should be mentioned that these arguments were not proven explicitly in \cite{CV1995} for the Laplacian case.
    \item Finally, we need to check the compact embedding properties of $\mathcal{W}^{2s, 1+\frac1p}(\Omega)$ in order to check the Palais-Smale condition.
\end{itemize}

To deal with these difficulties, we use properties of Legendre-Fenchel transform (see Lemma \ref{2302252143}) to ensure the well-definedness of $\mathcal{J}$. For  the differentiability of $\mathcal{J}$, we will prove in Lemma \ref{2211072137} that
\[
\mathcal{H}^*(f, g)=\int_{\Omega}H^*(f, g) dx,\quad \forall\; (f, g)\in X,
\]
where $H^*$ is the Legendre-Fenchel transform of $H: \mathbb{R}^2\to \mathbb{R}$.  We use the fact
$(\nabla H)^{-1}=\nabla H^*$ (see Lemma \ref{22112307}) which guarantees that $\mathcal{H}^*$ is well defined and of class $C^1$ over $X$. Moreover, under the superlinear growth assumption of $H$, $\mathcal{J}$ has a mountain pass geometry, and $H_z^*(f, g) := \nabla H^*(f, g)$ provides a weak solution to \eqref{main} given any critical point $(f, g)\in X$ for $\mathcal{J}$, see Proposition \ref{2211171427}. The compact embeddings of $\mathcal{W}^{2s, 1+\frac1p}(\Omega)$ are given in Proposition \ref{2302102144}.

\begin{theorem}\label{2211092215}
Let $N>2s$, $p, q>0$ satisfy \eqref{condition}. Assume that $H \in C^1(\mathbb{R}^2, \mathbb{R})$ satisfies
\begin{itemize}
\item[$(H5)$] There exist positive numbers $C_1, C_2$ such that
\[
\begin{split}
&C_1|u|^{p+1}\le H_u(u,v)u\le C_2\left( |u|^{p+1}+  |u|^{\alpha}|v|^{\beta}\right),\\
&C_1|v|^{q+1}\le H_v(u,v)v\le C_2\left( |v|^{q+1}+  |u|^{\alpha}|v|^{\beta}\right),
\end{split}
\]
 with
\begin{equation}\label{2402051803}
\frac{\alpha}{p+1}+\frac{\beta}{q+1}=1,\;\; \alpha, \beta>1;
\end{equation}
\item[$(H6)$] $\nabla H$ is strictly monotone, i.e.
\[
\langle \nabla H(u_1, v_1)-\nabla H(u_2, v_2),\, (u_1,v_1)- (u_2,v_2)\rangle>0,\quad \mbox{for any disjoint } (u_1, v_1),\, (u_2, v_2)\in \mathbb{R}^2;
\]
\item[$(H7)$] There is $\theta \in (0, 1)$ and positive numbers $C_3$, $C_4$ such that
\[
\theta H_u(u,v)u+ (1- \theta)H_v(u,v)v- H(u,v)\ge C_3\Big(|u|^{p+1}+|v|^{q+1}\Big)-C_4.
\]
\end{itemize}
Then there exists a nontrivial classical solution to \eqref{main}.
\end{theorem}
\begin{remark}
$(H5)$ implies indeed $(H1)$ and $(H4)$.  $(H7)$ is useful to prove that every Palais-Smale sequence of $\mathcal{J}$ is bounded, see Lemma \ref{2211171422}. Without loss of generality, we can assume $H(0, 0)=0$. Otherwise, we replace $H(u, v)$ by $H(u, v)-H(0, 0)$.
\end{remark}

\begin{remark}
The nonlinearities $H$ in Theorems \ref{220516}, \ref{2211092215} both have subcritical and superlinear growth.
But some differences exist between the two families of assumptions. For example, let $H$ have the form
\begin{equation}\label{2402052125}
H_\varepsilon(u, v)=|u|^{p+1}+|v|^{q+1}+\varepsilon |u|^{\alpha}|v|^{\beta},
\end{equation}
where $p, q$ satisfy \eqref{condition}, and $\alpha, \beta$ satisfy \eqref{2402051803}. When \eqref{Condition_pq} holds true, $H_\varepsilon$ satisfies $(H1)$-$(H4)$ for all $\varepsilon > 0$. On the other hand, we need not \eqref{Condition_pq} in Theorem \ref{2211092215}, but the strictly convexity assumption $(H6)$ fails for $H_\varepsilon$ when $\varepsilon$ is sufficiently large. In other words, Theorem \ref{220516} holds for more broad coupling nonlinearities but requires narrow choices of $p, q$; Theorem \ref{2211092215} can work for all subcritical and superlinear $p, q$, meanwhile the strict convexity is more restrictive for the coupling term.
\end{remark}

\subsection{Fractional Lane-Emden system}
As a special example, we revisit the fractional Lane-Emden system
\begin{equation}\label{2305191141}
\begin{cases}
\begin{aligned}
(-\Delta)^{s} u&=|v|^{q-1}v \,&&\text{in}~\Omega,\\
(-\Delta)^{s} v&=|u|^{p-1}u \,&&\text{in}~\Omega,\\
u &= v = 0 &&\text{in} ~ \mathbb{R}^N\setminus\Omega,\\
\end{aligned}
\end{cases}
\end{equation}
where $s\in (0, 1)$, $p, q\in (0, \infty)$, $N>2s$. As for the  classical Laplacian case, we consider subcritical exponents $p, q$, that is
\begin{equation}\label{2305191926}
p, q>0,\quad\frac{1}{p+1}+\frac{1}{q+1}>\frac{N-2s}{N}.
\end{equation}

Leite and Montenegro \cite{LM2019} showed the existence of positive viscosity solutions to \eqref{2305191141} under the subcritical condition \eqref{2305191926}, they reduced \eqref{2305191141} into a single equation and consider energy functional on $W_0^{1, 1+\frac1q}(\Omega)\cap W^{2s, 1+\frac1q}(\Omega)$.
 Choi and Kim \cite{CK2019, Choi2015} studied the related problems with respect to spectral fractional Laplacian.

Different from \cite{LM2019}, we work with the functional space $\mathcal{W}^{2s, 1+\frac1q}(\Omega)$.
Comparing with \cite{LM2019}, our setting avoid many regularity problems since every $\mathcal{W}^{2s, 1+\frac1q}(\Omega)$ solution is naturally a $L^1$-weak solution. More precisely, we consider
\begin{equation}\label{2305282030}
\begin{cases}
\begin{aligned}
(-\Delta)^{s}\left(|(-\Delta)^{s}u|^{\frac1q-1}(-\Delta)^{s}u\right) &=|u|^{p-1}u  \;&&\text{in}~\Omega,\\
u &=(-\Delta)^{s}u=0  &&\text{in} ~ \mathbb{R}^N\setminus\Omega.\\
\end{aligned}
\end{cases}
\end{equation}

\begin{definition}
\label{fenergysol}
We call $u$ an energy solution of \eqref{2305282030}, if $u\in \mathcal{W}^{2s, 1+\frac1q}(\Omega)$ and
\begin{equation}\label{2305282038}
\int_{\Omega}|(-\Delta)^{s}u|^{\frac1q-1}(-\Delta)^{s}u(-\Delta)^{s}\varphi dx=\int_{\Omega}|u|^{p-1}u\varphi dx,\quad \forall\,\varphi\in \mathcal{W}^{2s, 1+\frac1q}(\Omega).
\end{equation}
\end{definition}
Note that the above terms are well defined since $(-\Delta)^{s}u\in L^{1+\frac{1}{q}}(\Omega)$ for any $u\in \mathcal{W}^{2s, 1+\frac1q}(\Omega)$. Putting $v=|(-\Delta)^{s}u|^{\frac1q-1}(-\Delta)^{s}u$, we can prove that $(u, v)$ is a classical solution to \eqref{2305191141} if and only if $u$ is an energy solution to \eqref{2305282030}, see Proposition \ref{2306130008}.

Moreover, solutions to \eqref{2305282030} coincide with critical points of the following $C^1$ functional in $\mathcal{W}^{2s, 1+\frac1q}(\Omega)$,
\[
\mathcal{I}(u) :=\frac{q}{q+1}\left|(-\Delta)^su\right|_{1+\frac1q}^{1+\frac1q}-\frac1{p+1}|u|_{p+1}^{p+1}.
\]
Consider the Nehari manifold associated to $\mathcal{I}$,
\begin{equation}\label{2406220952}
\mathcal{N}_{\mathcal{I}}:=\{u\in \mathcal{W}^{2s, 1+\frac1q}(\Omega)\backslash \{0\}: \langle \mathcal{I}'(u), u\rangle=0\},
\end{equation}
and the ground state level is defined as 
\begin{equation}\label{2306152335}
c_{\mathcal{I}}:=\underset{u\in \mathcal{N}_{\mathcal{I}}}{\inf}\mathcal{I}(u).
\end{equation}
We establish the existence of positive solutions to \eqref{2305191141} by showing that $c_{\mathcal{I}}$ can be attained. 
\begin{theorem}\label{2306122224}
Assume that $s\in (0, 1)$, $N>2s$, $p, q$ satisfy \eqref{2305191926} and $pq\neq 1$. Then
\begin{itemize}
\item[\rm (i)] \eqref{2305191141} admits a positive classical ground state solution;
\item[\rm (ii)] If $pq<1$, the positive classical solution of \eqref{2305191141} is unique;
\item[\rm (iii)] If $\Omega$ is a ball, then \eqref{2305191141} has a positive radially symmetric classical solution.
\end{itemize}
\end{theorem}

When $\Omega$ is a ball and $pq<1$, the uniqueness of positive classical solution to \eqref{2305191141} ensures that the solution is radially symmetric. For $pq>1$, we can not claim directly as in \cite{BMR2012} that the ground state solution is radially symmetric, but we can work simply with $\mathcal{W}_{{\rm rad}}^{2s, r}(B_R)$, the subset of radial functions in $\mathcal{W}^{2s, r}(B_R)$. In fact, for $f\in C(\overline{B_R})$, $f\ge 0$, if $u$, $w$ are respectively solutions of
 \begin{equation*}
\begin{cases}
\begin{aligned}
(-\Delta)^{s} u&=f &&\text{in}~ B_R,\\
u&=0~&&\text{in} ~ \mathbb{R}^N\setminus\overline{B_R},\\
\end{aligned}
\end{cases}
\quad\quad
\begin{cases}
\begin{aligned}
(-\Delta)^{s} w&=f^{\#} &&\text{in}~B_R,\\
w&=0~&&\text{in} ~ \mathbb{R}^N\setminus\overline{B_R},\\
\end{aligned}
\end{cases}
\end{equation*}
where $f^{\#}$ denotes the radial decreasing rearrangement of $f$, we do not have always $u^{\#}\le w$, see \cite{FV2021}, while it is true for $s=1$. 


\begin{remark}
For any classical solution $(u, v)$ of \eqref{2305191141}, we have
\[
\int_{\Omega}|v|^{q+1} dx=\langle (-\Delta)^s u, v \rangle=\langle (-\Delta)^s v, u \rangle=\int_{\Omega}|u|^{p+1} dx.
\]
Combining with $\langle \mathcal{I}'(u), u\rangle=0$, there holds
$\mathcal{K}(u, v)=\mathcal{I}(u)$. Hence $\mathcal{K}(u, v)$ reaches the minimal energy among all classical solutions if $u$ is a ground state of \eqref{2305282030}.
\end{remark}


The paper is organized as follows. In section \ref{2306052054}, we introduce some preliminary results. In section \ref{2306141421}, we construct an interpolation space setting and prove Theorem \ref{220516}. In section \ref{2306141749}, we construct the framework of dual method to show Theorem \ref{2211092215}. In section \ref{2306081126}, we revisit the fractional Lane-Emden system \eqref{2305191141}.

\section{Notations and preliminary}\label{2306052054}
Throughout this paper, we use the following notations.
\begin{itemize}
\item We denote by $\nabla H$ the gradient of $H$ in $\R^2$, and we write $H_z(u, v) := \nabla H(u(x), v(x))$ for the functional variable case where $z=(u, v)$. The same convention is used for the Legendre-Fenchel transform $H^*$.
\item $C, C', C_1, C_2,...$ denote always generic positive constants.
\item For $r\in [1, \infty]$, we denote by $|u|_r$ the usual $L^r(\Omega)$ or $L^r(\R^N)$ norm.
\item For any $p, q > 0$, we set $X:=L^{1+\frac1p}(\Omega)\times L^{1+\frac1q}(\Omega)$ with norm $\|(f, g)\|_{X}=|f|_{1+\frac1p}+|g|_{1+\frac1q}$, hence its dual space is $X^*=L^{p+1}(\Omega)\times L^{q+1}(\Omega)$.
\item The Hamiltonian functional is denoted by
\[
\mathcal{H}(u, v) =\int_{\Omega}H(u, v) dx,\quad (u, v)\in X^*.
\]
\end{itemize}


\subsection{Basic properties for $L^1$-weak solutions}
Here we show some elementary facts for $L^1$-weak solutions. Many of them were established for solutions in $X_0^s(\Omega)$, however it's worthy to check carefully for $L^1$-weak solutions. 

Let us begin with the comparison principle. We say that
 $(-\Delta)^su \le (-\Delta)^sv$ in the $L^1$-weak sense, if $u, v\in L^1(\Omega)$ and
 \begin{equation}\label{2306011618}
 \int_{\Omega}(u-v)(-\Delta)^s\varphi dx\le 0,\quad \forall\,\varphi\in \mathcal{T}_s(\Omega) \mbox{ with } (-\Delta)^s\varphi\ge 0.
 \end{equation}
By the definition of $\mathcal{T}_s(\Omega)$, as $(-\Delta)^s\varphi$ fulfilled $C_c^\infty(\Omega)$, we get immediately
\begin{lemma}\label{2306081933}
If $(-\Delta)^su \le (-\Delta)^sv$ in the $L^1$-weak sense, then $u\le v$ a.e.~in $\Omega$.
\end{lemma}

The next one is a Hopf lemma result for $L^1$-weak supersolutions.
\begin{lemma}\label{2306012122}
Assume that $f\in L^1(\Omega; \delta^s dx)$ be nonnegative and $f\ne 0$. Then there exists $C> 0$ depending only on $\Omega$, $f$ and $s$ such that for any $u\in L^1(\Omega)$ satisfying $(-\Delta)^s u\ge f$, $u>C\delta^s$ a.e.~in $\Omega$.
\end{lemma}
\begin{proof}
Fix $f_k:=\min\{k, f\}$ with $k$ large enough such that $f_k\ne 0$. Let $u_k \subset X_0^s(\Omega)$ solve (ensured by Lemma \ref{2402061816})
\begin{equation}
\label{2304142022}
(-\Delta)^s u_k = f_k \;\;\text{in}~\Omega,\quad u_k =0 \;\;\text{in}~\mathbb{R}^N\setminus\Omega.
\end{equation}
As $f_k \in L^\infty(\Omega)$, by maximum principle (see \cite[Proposition 2.2.8]{S2007}, \cite[Theorem A.1]{BF2014}) and Hopf's lemma (see \cite[Proposition 2.7]{CRS2010}, \cite[Lemma 3.2]{RS2014} and \cite[Theorem 1.5]{DQ2017}) for bounded source, there holds $u_k>C\delta^s$ for $C>0$. Using previous lemma, we get $u\ge u_k$, and finish the proof.
\end{proof}

Inspired by \cite{BCM1996}, we show the following existence and uniqueness result for $L^1$-weak solution.
\begin{proposition}\label{2304232341}
Given any $f\in L^1(\Omega; \delta^s dx) $, there exists a unique $L^1$-weak solution $u$ to \eqref{2306011244}. Moreover, there exists $C=C(s, \Omega)>0$ such that
\begin{equation}\label{2304142025}
|u|_1\le C\|f\|_{L^1(\Omega; \delta^s dx) }.
\end{equation}
\end{proposition}
\begin{proof}
Assume first $f\ge 0$, otherwise we decompose $f = f_+ - f_-$. Let $\xi\in X_0^s(\Omega)$ be the solution of
\begin{equation}\label{2306151956}
(-\Delta)^s \xi = 1 \;\;\text{in}~\Omega,\quad
\xi =0 \;\;\text{in}~\mathbb{R}^N\setminus\Omega.
\end{equation}
For any $k\in \N$, set $f_k:=\min\{k, f\}$, and $u_k\geq 0$ the solution of \eqref{2304142022}.
As $u_k \in X_0^s(\Omega)$, we can use it as test function to \eqref{2306151956}. Applying the estimates in Lemma \ref{2402061816} for $\xi$,
\[
|u_k|_1 = \langle (-\Delta)^s \xi, u_k\rangle = \langle (-\Delta)^s u_k, \xi\rangle = \int_\O f_k\xi dx \leq C\int_\O f_k\delta^sdx \leq C\|f\|_{L^1(\Omega; \delta^s dx)}.
\]
Similarly, considering $u_k - u_l$, we have
$|u_k-u_l|_{L^1} \leq C\|f_k-f_l\|_{L^1(\Omega; \delta^s dx)}$, so $\{u_k\}$ is a Cauchy sequence, hence a convergent sequence in $L^1(\O)$. We get easily a $L^1$-weak solution to \eqref{2306011244} by taking $u$, the limit of $u_k$ in $L^1(\Omega)$. The uniqueness is ensured by the comparison principle.
\end{proof}

\begin{remark}\label{2402062112}
If $\Omega = B_R$ is a ball and $f\in L^1(B_R, \delta^s dx)$ is radial, then there exists a unique radially symmetric solution to \eqref{2306011244}.
\end{remark}

Next, we state the regularity results for the unique $L^1$-weak solution when $f$ admits further integrability, here we summarize the results in \cite{LPPS2015} and \cite[Lemma 2.5]{BWZ2017}.
\begin{proposition}\label{2211162013}
Assume that $s\in (0, 1)$, $N>2s$, $\Omega$ is a smooth bounded domain. For any $f\in L^r(\Omega)$ with $r\ge 1$, the unique $L^1$-weak solution $u$ to \eqref{2306011244} satisfies
\begin{itemize}
\item[\rm (i)] If $1 \le r < \frac{N}{2s}$, $u\in L^{\gamma}(\Omega)$ and $|u|_{\gamma}\le C(\Omega, r, s)|f|_r$,
where
\[
\gamma={\frac{Nr}{N-2rs}}\; {\rm if }\; r>1, \quad 1\le \gamma<\frac{N}{N-2s}\;{\rm if }\; r=1;
\]
\item[\rm (ii)] If $r \ge \frac{2N}{N+2s}$, we have $u\in X_0^s(\Omega)$ and $\|u\|_{X_0^s(\Omega)}\le C(\Omega, r, s)|f|_r$;
\item[\rm (iii)] If $r>\frac{N}{2s}$, there holds $u\in L^{\infty}(\Omega)$ and $|u|_{\infty}\le C(\Omega, r, s)|f|_{r}$.
\item[\rm (iv)] If $r=\frac{N}{2s}$, there is a constant $\alpha(\O, f, s) >0$ such that
\[
\int_{\Omega} e^{\alpha |u|} dx<\infty.
\]
In particular, $u\in L^\gamma(\Omega)$ for all $1\le \gamma<\infty$.
\end{itemize}
\end{proposition}

\subsection{Embedding and density}
Here we expose some basic properties of $\mathcal{W}^{2s, r}(\Omega)$, including (compact) embedding results and density property, which will be important for constructing the variational framework.

\begin{proposition}\label{2302102144}
Assume that $s\in (0,1)$, $N>2s$, $r\ge 1$, $\Omega$ is a smooth bounded domain. Then $\mathcal{W}^{2s, r}(\Omega)$ is a Banach space and $\mathcal{T}_s(\Omega)$ is dense in $\mathcal{W}^{2s, r}(\Omega)$. On the other hand, we have continuous embedding $\mathcal{W}^{2s, r}(\Omega)\subset L^{\gamma}(\Omega)$ for
\begin{equation}\label{2302110107}
\begin{cases}
\begin{aligned}
&1\le \gamma<\frac{N}{N-2s}, \;\; &&\text{if}~ r=1;\\
&1\le \gamma\le\frac{Nr}{N-2sr}, \;\; &&\text{if}~1 < r < \frac{N}{2s};\\
&1\le \gamma<\infty,  && \text{if}~2sr = N;\\
& \gamma = \infty,  && \text{if}~ 2sr > N.\\
\end{aligned}
\end{cases}
\end{equation}
Moreover, the above embeddings are compact provided
\begin{equation}\label{2302110109}
1\le \gamma<\frac{Nr}{N-2sr}\;\;  \text{if}~N>2sr;\quad \mbox{or} \;\; 1\le \gamma<\infty, \;\; \text{if}~N\le 2sr.
\end{equation}
\end{proposition}

\begin{proof}
The density of $\mathcal{T}_s(\Omega)$ in $\mathcal{W}^{2s, r}(\Omega)$ is obvious by definition and Lemma \ref{2402061816}. 

The completeness of $\mathcal{W}^{2s, r}(\Omega)$ is a simple fact. Let $\{u_j\}\subset \mathcal{W}^{2s, r}(\Omega)$ be a Cauchy sequence, by definition, $\{(-\Delta)^s u_j\}$ is a Cauchy sequence in $L^r({\Omega})$, hence converges to $v\in L^r({\Omega})$. Using Proposition \ref{2304232341}, there is a unique $L^1$-weak solution $u$ satisfying \eqref{2306011244} with $f = v$. Clearly $u \in \mathcal{W}^{2s, r}(\Omega)$, and $u_j$ tends to $u$ in $\mathcal{W}^{2s, r}(\Omega)$.

On the other hand, in virtue of Proposition \ref{2211162013}, the embedding of $\mathcal{W}^{2s, r}(\Omega)\subset L^\gamma(\Omega)$ is continuous if \eqref{2302110107} holds true, we shall only prove that the embedding is compact provided \eqref{2302110109}.

Without loss of generality, we only consider the case $N>2sr$, $1 < r$, $1 \le\gamma<\frac{Nr}{N-2sr}$. If $r\ge \frac{2N}{N+2s}$, then $\mathcal{W}^{2s, r}(\Omega)\subset X_0^s(\Omega)$ is continuous by Proposition \ref{2211162013}. By H\"older inequality, fix any $1 < \gamma<\frac{Nr}{N-2sr}$,
\begin{equation}\label{2302111630}
|u|_{{\gamma}}\le |u|^{1-\theta}_{{1}}|u|^{\theta}_{{\frac{Nr}{N-2sr}}} \quad \mbox{with } \theta=\frac{Nr-\gamma (N-2sr)}{Nr-(N-2sr)}.
\end{equation}
Since $\mathcal{W}^{2s, r}(\Omega)$ is continuously embedded in $L^{\frac{Nr}{N-2sr}}(\Omega)$ and compactly embedded  in $L^1(\Omega)$, $\mathcal{W}^{2s, r}(\Omega)$ is compactly embedded in $L^{\gamma}(\Omega)$.

Suppose now $1<r< \frac{2N}{N+2s}$, we first claim that the embedding $\mathcal{W}^{2s, r}(\Omega)\subset L^{1}(\Omega)$ is compact. Indeed, let $\{u_j\}$ be a bounded sequence in $\mathcal{W}^{2s, r}(\Omega)$, then $\{u_j\}$ is bounded in $L^{\frac{Nr}{N-2sr}}(\Omega)$ and the sequence $f_j = (-\Delta)^{s} u_j$ is bounded in $L^r(\Omega)$. Let $A > 0$ to be chosen later, we denote $g_j = f_j\chi_{|f_j(x)|\le A}$ and $h_j = f_j\chi_{|f_j(x)|> A}$. By H\"older and Chebychev inequalities,
\begin{align*}
\|h_j\|_{L^1(\Omega; \delta^s dx)} \leq C|h_j|_1\leq C|f_j|_r\Big(\int_{|f_j(x)|> A} dx\Big)^{1-\frac{1}{r}} \le C'A^{\frac{1}{r}-1}.
\end{align*}
Let $\{v_j\}$ and $\{w_j\}$ be respectively $L^1$-weak solutions of
\begin{equation}\label{2306010025}
\left\{
\begin{split}
(-\Delta)^{s} v_j &=g_j &&\text{in}~\Omega,\\
v_j&=0 &&\text{in} ~ \mathbb{R}^N\setminus\Omega,\\
\end{split}
\right. \quad \left\{
\begin{split}
(-\Delta)^{s} w_j&=h_j &&\text{in}~\Omega,\\
w_j&=0 &&\text{in} ~ \mathbb{R}^N\setminus\Omega.\\
\end{split}
\right.
\end{equation}
Clearly $u_j=v_j+w_j$. Given any $\varepsilon>0$, we fix $A$ large enough such that 
\[
|w_j|_1\le C\|h_j\|_{L^1(\Omega; \delta^s dx)} \le C'A^{\frac{1}{r}-1} \le \varepsilon.
\]
Moreover, $\{v_j\}$ is bounded in $X_0^s(\Omega)$ hence relatively compact in $L^1(\Omega)$. Therefore $\{u_j\}$ is precompact, or equivalently relatively compact in $L^1(\Omega)$, which means $\mathcal{W}^{2s, r}(\Omega)$ is compactly embedded in $L^{1}(\Omega)$. Applying again the interpolation inequality \eqref{2302111630}, $\{u_j\}$ is compact in $L^\gamma(\Omega)$ for $1\le\gamma<\frac{Nr}{N-2sr}$. So we are done.
\end{proof}

\subsection{A linking theorem} The Linking theorem \ref{2210091640} below was given by Felmer \cite[Theorem 3.1]{Felmer1993}, it is useful for the proof of Theorem \ref{220516}. For the sake of completeness, we recall also the definition of Palais-Smale condition.
\begin{definition}
For a Banach space $X$,  ${\mathcal I}\in C^1(X, \mathbb{R})$ is said satisfying the Palais-Smale condition at level $c\in \mathbb{R}$ {\rm (}for short $(PS)_c${\rm )} if any sequence $\{u_j\}\subset X$ satisfying
\begin{equation*}\label{eq2.4}
{\mathcal I}(u_j)\rightarrow c \quad \mbox{and}\quad {\mathcal I}'(u_j) \rightarrow 0 \; \mbox{ in }\;  X^*
\end{equation*}
admits a convergent subsequence in $X$. We say that ${\mathcal I}$ satisfies the Palais-Smale condition {\rm (}$(PS)$ for short{\rm )} if $(PS)_c$ is satisfied for all $c\in \R$.
\end{definition}

\begin{theorem}\label{2210091640}
Let $(H, \langle \cdot, \cdot \rangle)$ be a Hilbert space such that $H=H_1\oplus H_2$. Suppose that $\mathcal{I}\in C^1(H)$ satisfies $(PS)$ condition and $\mathcal{I}(z)=\frac12 \langle Lz, z\rangle-\mathcal{J}(z)$, where
\begin{itemize}
\item[\rm (i)] $L : H\to H$ is a bounded, self-adjoint linear operator and $L(H_1) \subset H_1$; $\mathcal{J}' : H\to H^*$ is compact;
\item[\rm (ii)] There exist two linear, bounded, invertible operators $B_1$, $B_2 : H\to H$ such that $\widehat{B}_{\tau}=P_2B_1^{-1}e^{\tau L}B_2: H_2\to H_2$  is invertible for any $\tau\ge 0$. Here $P_2: H\to H_2$ is the projection along $H_1$ and $e^{\tau L}=\sum_{n\in \N}\frac{(\tau L)^n}{n!}$;
\item[\rm (iii)] Let $e_1\in H_1$ with $\|e_1\|=1$. Let $\rho>0$, $R_1>\frac{\rho}{\|B_1^{-1}B_2 e_1\|}$, $R_2>\rho$ and define
$$
S =\{B_1z_1 : z_1\in H_1, \|z_1\|=\rho\},\quad
Q =\{B_2(te_1+z_2) : 0\le t\le R_1, z_2\in H_2, \|z_2\|\le R_2\}.
$$
Suppose that $\mathcal{I}(z)\ge \sigma >0$ on $S$ and $\mathcal{I}(z)\le 0$ on $\partial Q$.\end{itemize}
Then $\mathcal{I}$ has a critical point $z_0 \in H$ such that $\mathcal{I}(z_0)\ge \sigma$.
\end{theorem}

\subsection{Legendre-Fenchel transform}
Let $V$ be a Banach space and $V^*$ be its dual space. For a function $G: V\to \mathbb{R}\cup \{+\infty\}$, $G\not\equiv +\infty$, the function $G^*: V^*\to \mathbb{R}\cup \{+\infty\}$ given by
\begin{equation}\label{2305291511}
G^*(u^*)=\sup\{\langle u^*, v \rangle-G(v): v\in V\},\quad \forall\, u^*\in V^*,
\end{equation}
is named the Legendre-Fenchel transform of $G$. The following are some basic properties of $G^*$.
\begin{lemma}\label{2302252143}
Let $G^*$ be the Legendre-Fenchel transform of $G$.
\begin{itemize}
\item[\rm (a)] The Legendre-Fenchel transform reverses the order, that is, $\widetilde{G}\ge G$ implies $\widetilde{G}^*\le G^*$.
\item[\rm (b)] If $G\in C^1(V, \mathbb{R})$ is convex and $G^*\in C^1(V^*, \mathbb{R})$, then for $v\in V$, $u^*\in V^*$, there holds
\[
G^*(u^*)+G(v)=\langle u^*, v\rangle \iff v=(G^*)'(u^*) \iff u^*=G'(v).
\]
\item[\rm (c)] If $V=\mathbb{R}^N$, $G$ is lower semi-continuous, strictly convex, and $\lim_{|x|\to \infty}\frac{G(x)}{|x|}=+\infty$, then $G^*\in C^1(\mathbb{R}^N, \mathbb{R})$.
\item[\rm (d)] If $V=\mathbb{R}^N$, $G(x)=\frac1{p+1}|x|^{p+1}$ with $p>0$, then  $G^*(x)=\frac{p}{p+1} |x|^{1+\frac1p}$. If $V=L^{p+1}(\Omega)$ with $p>0$, $G(v)=\frac{1}{p+1}|v|_{p+1}^{p+1}$, then $G^*(u^*)=\frac{p}{p+1} |u^*|^{1+\frac1p}_{1+\frac1p}$.
\item[\rm (e)] Let $a>0$ and $G_a(u):=aG(u)$, then
$G^*_a(u^*)=aG^*\big(\frac{u^*}{a}\big)$.
\end{itemize}
\end{lemma}
\begin{proof}
For ${\rm (a)}$-${\rm (d)}$, see respectively \cite[Chapter I, Section 6.2]{Struwe2008}, \cite[Chapter I, Lemma 6.3]{Struwe2008}, \cite[Proposition 2.4]{MW1989} and \cite[Proposition 4.2, Remark 4.1]{ET1999}. For ${\rm (e)}$, one can check directly by means of \eqref{2305291511}.
\end{proof}

\section{Proof of Theorem \ref{220516}}\label{2306141421}
In this section, we will use interpolation space method to handle Theorem \ref{220516}. Throughout this section, we assume that $H$ satisfies $(H1)$-$(H4)$. 

Let us begin with the understanding of interpolation spaces $E^{\alpha}$ ($\alpha \ge 0$) and interpolation operators $A^{\alpha}$ given in \eqref{2406191546} and \eqref{2406191601}. First, $E^0=L^2(\Omega)$ and $A^{0} = {\rm id}_{L^2(\O)}$. Clearly for $\alpha \ge 0$, $A^{\alpha}$ is an isomorphism and its inverse $A^{-\alpha}: L^2(\Omega)\to E^{\alpha}$ can be denoted by
\[
A^{-\alpha}u=\sum_{j\ge 1}\lambda_j^{\frac{-\alpha}{2s}}a_j\varphi_j\quad \forall\, u=\sum_{j\ge 1}a_j\varphi_j\in L^2(\Omega).
\]
More generally, for convenience, we define formally, when it makes sense, 
\begin{align}
\label{Abeta}
A^\beta\Big(\sum_{j\ge 1}a_j\varphi_j\Big) =\sum_{j\ge 1}\lambda_j^{\frac{\beta}{2s}}a_j\varphi_j, \quad \forall\;\beta \in \R.
\end{align}
\begin{proposition}\label{2209072124}
For any $s \in (0, 1)$, there hold $E^{2s}=\mathcal{H}^{2s}(\Omega)$, and $A^{2s}u=(-\Delta)^s u$ for all $u\in E^{2s}$.
\end{proposition}
\begin{proof}
For any $u\in \mathcal{H}^{2s}(\Omega)$, let
\[
u=\sum_{j\ge 1} a_j \varphi_j\in L^2(\Omega), \quad (-\Delta)^s u=\sum_{j\ge 1} b_j \varphi_j\in L^2(\Omega).
\]
Observe that
\[
b_j=\int_{\Omega}(-\Delta)^s u \varphi_j dx=\int_{\Omega}  (-\Delta)^s \varphi_j udx=\int_{\Omega}  \lambda_j \varphi_j udx=a_j\lambda_j.
\]
So $ (-\Delta)^s u= \sum_{j\ge 1} a_j \lambda_j \varphi_j$ and $\sum_{j\ge 1}a^2_j\lambda^2_j<\infty$ as $u\in \mathcal{H}^{2s}(\Omega)$. It follows that $\mathcal{H}^{2s}(\Omega)\subset E^{2s}$ and $(-\Delta)^s u = A^{2s}u$ for all $u\in \mathcal{H}^{2s}(\Omega)$.

Conversely, let $u\in E^{2s}$, then $u=\sum_{j\ge 1} a_j\varphi_j$ with $\sum_{j\ge 1} \lambda_j^2a_j^2 < \infty$. Let $\varphi \in \mathcal{H}^{2s}(\Omega)$, $\varphi=\sum_{j\ge 1} b_j\varphi_j$, we consider
\begin{equation}\label{2306142028}
f(\varphi)=\int_{\Omega}u (-\Delta)^s\varphi dx = \sum_{j\ge 1} a_jb_j\lambda_j.
\end{equation}
Therefore
\[
|f(\varphi)| \le \Big|\sum_{j\ge 1} a_j^2\lambda_j^2\Big|^{\frac12}\Big|\sum_{j\ge 1} b_j^2\Big|^{\frac12}= \lVert u\rVert_{E^{2s}}| \varphi|_{2}.
\]
Using Riesz's representation theorem and the density of $\mathcal{H}^{2s}(\Omega)$ in $L^2(\Omega)$ (all eigenfunctions belong to $\mathcal{H}^{2s}(\Omega)$), we deduce that there exists $v\in L^2(\Omega)$ such that
\begin{equation}\label{2306142029}
f(\varphi)=\int_{\Omega}v\varphi dx.
\end{equation}
Combining \eqref{2306142028} and \eqref{2306142029}, we get $(-\Delta)^su=v$, so $(-\Delta)^su\in L^2(\Omega)$ and $u\in \mathcal{H}^{2s}(\Omega)$.
\end{proof}
\begin{remark}\label{2211231543}
By definition of $E^\alpha$, if $0<\alpha<2s$, $E^\alpha$ is the real interpolation space $[L^2(\Omega), E^{2s}(\Omega)]_{\alpha / 2s}$. Therefore, applying Proposition \ref{2302102144} and Proposition \ref{2209072124}, $E^\alpha \subset L^{p+1}(\Omega)$ is a continuous embedding when $\frac{1}{p+1}\ge \frac12-\frac{\alpha}{N}$ and this embedding is compact provided that the strict inequality is valid, see \cite{Persson1964} and \cite[Sections 7.22, 7.23]{Adams2003}.
\end{remark}

By  \eqref{condition}, \eqref{Condition_pq} and Remark \ref{2211231543},  fix $\alpha\in (0, 2s)$ such that
\begin{equation}\label{2211231541}
N\left(\frac12-\frac{1}{\max(p,q)+1}\right)<\alpha<N\left(\frac{1}{\min(p, q)+1}-\frac{N-4s}{2N}\right),
\end{equation}
so that
$$\mathbf{E}_{\alpha} :=E^{\alpha}\times E^{2s-\alpha}\subset X^*$$
and $\mathcal{E}: \mathbf{E}_{\alpha}\to \mathbb{R}$ given in \eqref{2306171702} is well defined.  In this section, we always consider $\mathbf{E}_{\alpha}$ with $\alpha \in (0, 2s)$ satisfying \eqref{2211231541}. $\mathbf{E}_{\alpha}$ is a Hilbert space with scalar product
\[
( (u_1, v_1), (u_2, v_2))_{\mathbf{E}_{\alpha}}=(u_1, u_2)_{E^\alpha}+(v_1, v_2)_{E^{2s-\alpha}}
\]
and $\mathbf{E}_{\alpha}=\mathbf{E}^+_{\alpha}\oplus \mathbf{E}^-_{\alpha}$, where
\[
\begin{split}
\mathbf{E}^+_{\alpha}&=\big\{(u, A^{-2s+\alpha}\circ A^{\alpha} u): u\in E^\alpha\big\},\quad \mathbf{E}^-_{\alpha}=\big\{(u, -A^{-2s+\alpha}\circ A^{\alpha} u): u\in E^\alpha\big\}.
\end{split}
\]
In the spirit of \eqref{Abeta}, we will write $A^{-2s+\alpha}\circ A^{\alpha}$ as $A^{-2s+2\alpha}$ for briefness, similarly we write $A^{2s-2\alpha}$ for $A^{-\alpha}\circ A^{2s-\alpha}$. If $z=(u,v)\in \mathbf{E}_{\alpha}$, denote
\[
\begin{split}
z^+&=(u^+,v^+)=\left(\frac{u+A^{2s-2\alpha}v}{2},\frac{v+A^{-2s+2\alpha}u}{2} \right)\in \mathbf{E}^+_{\alpha},\\
z^-&=(u^-,v^-)=\left(\frac{u-A^{2s-2\alpha}v}{2},\frac{v-A^{-2s+2\alpha} u}{2}\right)\in \mathbf{E}^-_{\alpha},
\end{split}
\]
then $z=z^++z^-$. Notice that although the quadratic part of $\mathcal{E}$ turns to be strongly indefinite, it is positive definite in $\mathbf{E}^+_{\alpha}$ and negative definite in $\mathbf{E}^-_{\alpha}$.  Consider the bilinear form $B: \mathbf{E}_{\alpha} \times \mathbf{E}_{\alpha} \to \mathbb{R}$ defined by
\begin{equation}\label{2401062223}
B((u,v), (\varphi, \psi))=\int_{\Omega}A^\alpha uA^{2s-\alpha}\psi\, +A^{2s-\alpha} vA^{\alpha}\varphi\, dx.
\end{equation}
It is easy to see that $B$ is continuous and symmetric, which induces a self-adjoint bounded linear operator $L: \mathbf{E}_{\alpha}\to \mathbf{E}_{\alpha}$ satisfying
\begin{equation}\label{2401131356}
B((u, v), (\varphi, \psi))=\big\langle L(u, v), (\varphi, \psi)\big\rangle_{\mathbf{E}_{\alpha}},\quad\forall\; (u, v), (\varphi, \psi)\in \mathbf{E}_{\alpha}.
\end{equation}

\begin{remark}\label{2302211357}
By definition, we see easily that
\[
L(u, v)=(A^{2s-2\alpha}v, A^{2\alpha-2s}u ),\quad \forall\; (u, v)\in \mathbf{E}_{\alpha},
\]
and $L|_{ \mathbf{E}_{\alpha}^+}={\rm id}_{\mathbf{E}_{\alpha}^+}$, $L|_{ \mathbf{E}_{\alpha}^-}=-{\rm id}_{\mathbf{E}_{\alpha}^-}$.
\end{remark}

\begin{proposition}\label{2302302237}
Let $p$, $q$ satisfy \eqref{condition}-\eqref{Condition_pq}, and $(u,v)\in \mathbf{E}_{\alpha}$ be a critical point of $\mathcal{E}$. Then $(u, v)$ is a $L^1$-weak solution to \eqref{main}.
\end{proposition}
\begin{proof}
Since $(u,v)\in \mathbf{E}_{\alpha}$ is a critical point of $\mathcal{E}$, $\langle \mathcal{E}'(u,v), (\varphi,\psi)\rangle=0$ for all $(\varphi,\psi)\in \mathbf{E}_{\alpha}$, i.e.
\[
\int_{\Omega}A^\alpha uA^{2s-\alpha}\psi\, dx+\int_{\Omega}A^{2s-\alpha} vA^{\alpha}\varphi\, dx-\int_{\Omega}H_v(u,v)\psi\, dx-\int_{\Omega}H_u(u, v)\varphi\, dx=0.
\]
Taking $\varphi=0$, we get
\begin{equation}\label{2302202221}
\int_{\Omega}A^\alpha uA^{2s-\alpha}\psi\, dx=\int_{\Omega}H_v(u,v)\psi\, dx.
\end{equation}
It follows from Proposition \ref{2209072124} that $\mathcal{T}_s(\Omega)\subset E^{2s}\subset E^{2s-\alpha}$. Let $\psi\in \mathcal{T}_s(\Omega)$ with $\psi=\sum_{j\ge 1}  b_j \varphi_j$, and decompose $u=\sum_{j\ge 1} a_j \varphi_j$.  By Proposition \ref{2209072124},
\begin{equation}\label{2302202220}
\int_{\Omega}A^\alpha uA^{2s-\alpha}\psi\, dx=\sum_{j\ge 1}  a_jb_j\lambda_j^{s}=\int_{\Omega} uA^{2s}\psi\, dx=\int_{\Omega} u(-\Delta)^s\psi\, dx.
\end{equation}
Combining \eqref{2302202221} with \eqref{2302202220}, we obtain
$$\int_{\Omega} u(-\Delta)^s\psi\, dx=\int_{\Omega}H_v(u,v)\psi\, dx.$$
Similarly, for any $\varphi \in \mathcal{T}_s(\Omega)$, there holds
$$\int_{\Omega} v(-\Delta)^s\varphi\, dx=\int_{\Omega}H_u(u,v)\varphi\, dx.$$
So $(u,v)$ is a $L^1$-weak solution to (\ref{main}).
\end{proof}

To find a nontrivial critical point of $\mathcal{E}$ by Theorem \ref{2210091640}, our next step is to check the $(PS)$ condition for $\mathcal{E}$. We start from the following compactness property.
\begin{lemma}\label{2211151517}
Let $p$, $q$ satisfy \eqref{condition}-\eqref{Condition_pq}, ${H}$ satisfy $(H4)$ and $\alpha$ satisfy \eqref{2211231541}, then $H_z: \mathbf{E}_{\alpha}\to X$ is compact.
\end{lemma}
\begin{proof}
By $(H4)$ and Remark \ref{2211231543}, $H_z$ is well defined from $\mathbf{E}_{\alpha}$ to $X^*$. Let $\{(u_j, v_j)\}$ be bounded in $\mathbf{E}_{\alpha}$, we aim to prove that $\{(H_u(u_j, v_j), H_v(u_j, v_j))\}$ has a convergent subsequence in $X$.

Since $\mathbf{E}_{\alpha}$ is a Hilbert space, up to a subsequence, there exists $(u, v)\in \mathbf{E}_{\alpha}$ such that
$(u_j, v_j) \rightharpoonup (u, v)$ weakly in $\mathbf{E}_{\alpha}$. By virtue of the compact embeddings $E^\alpha\subset L^{p+1}(\Omega)$, $E^{2s-\alpha}\subset L^{q+1}(\Omega)$ (see Remark \ref{2211231543} and \eqref{2211231541}), there holds, up to a subsequence,
\[
u_j\to u ~\text{in}~L^{p+1}(\Omega), \;\; v_j\to v ~\text{in}~L^{q+1}(\Omega)  \quad\mbox{and}\quad u_j(x)\to u(x),\;\; v_j(x)\to v(x)\;\; {\rm a.e.~in }~\Omega..
\]
Using \cite[Theorem 4.9]{Brezis2010}, there exists $(\Phi_1, \Phi_2)\in X^*$ such that $|u_j|\le \Phi_1, |v_j|\le \Phi_2$ a.e.~for all $j$. Applying $(H4)$ again, we obtain
\[
\begin{aligned}
&|H_u(u_j, v_j)|\le C\big(| \Phi_1|^p+| \Phi_2|^\frac{p(q+1)}{(p+1)}+1\big)\in L^{1+\frac1p}(\Omega),\\
&|H_v(u_j, v_j)|\le C\big(| \Phi_2|^q+| \Phi_1|^\frac{q(p+1)}{(q+1)}+1\big)\in L^{1+\frac1q}(\Omega).
\end{aligned}
\]
Let
\[
w_j=H_u(u_j, v_j),\quad w=H_u(u, v),\quad y_j=H_v(u_j, v_j),\quad y=H_v(u, v).
\]
Using the Lebesgue dominated convergence theorem, we get $(w_j, y_j)\to (w, y)$ in $X$. The proof is done.
\end{proof}
\begin{remark}\label{2302211335}
Clearly,  under the assumptions of Lemma \ref{2211151517}, $\mathcal{H}': \mathbf{E}_{\alpha}\to \mathbf{E}_{\alpha}^*$ is compact, which can be seen as follows
\[
\mathbf{E}_{\alpha}\underset{H_z}{\to} X\underset{\rm id}{\to} \mathbf{E}_{\alpha}^*.
\]
The first mapping above is compact by Lemma \ref{2211151517}, and the last one is continuous by the embedding $\mathbf{E}_{\alpha}\subset X^*$.
\end{remark}

\begin{lemma}\label{2211161003}
Let $p$, $q$ satisfy \eqref{condition}-\eqref{Condition_pq}, ${H}$ satisfy $(H2)$ and $(H4)$. Then the functional $\mathcal{E}$ satisfies $(PS)_c$ condition for any $c\in \mathbb{R}$.
\end{lemma}
\begin{proof}
Let $\{(u_j, v_j)\}\subset \mathbf{E}_{\alpha}$ be any $(PS)_c$ sequence of $\mathcal{E}$ with $c\in\mathbb{R}$. By definition, when $j$ goes to infinity,
 \begin{equation}\label{eq3.1}
\mathcal{E}(u_j,v_j)=\int_{\Omega}A^{\alpha}u_jA^{2s-\alpha}v_j\, dx-\int_{\Omega}H(u_j, v_j)\, dx=c+o(1),
 \end{equation}
 and for all $(\varphi, \psi)\in \mathbf{E}_{\alpha}$,
\begin{equation}\label{eq3.2}
\begin{aligned}
\langle \mathcal{E}'(u_j, v_j), (\varphi, \psi)\rangle& =\int_{\Omega}A^{\alpha}u_jA^{2s-\alpha}\psi \, dx+\int_{\Omega}A^{2s-\alpha}v_jA^{\alpha}\varphi\, dx-\int_{\Omega}H_u(u_j, v_j) \varphi\, dx\\
&\quad -\int_{\Omega}H_v(u_j, v_j) \psi dx\\
& =o(1)\lVert (\varphi, \psi) \lVert_{\mathbf{E}_{\alpha}}.
\end{aligned}
\end{equation}
We first claim that $\{(u_j, v_j)\}$ is bounded in $\mathbf{E}_{\alpha}$. Take first
$$(\varphi, \psi)=\left(\frac{q+1}{p+q+2}u_j, \frac{p+1}{p+q+2}v_j\right).$$
From (\ref{eq3.2}), (\ref{eq3.1}) and $(H2)$, as $j$ goes to infinity,
\[
\begin{split}
& \quad c+o(1)\|(u_j, v_j)\|_{\mathbf{E}_{\alpha}}+o(1)\\
& = \mathcal{E}(u_j,v_j)-\left\langle \mathcal{E}'(u_j, v_j), \left(\frac{q+1}{p+q+2}u_j, \frac{p+1}{p+q+2}v_j\right)\right\rangle\\
& = \frac{(p+1)(q+1)}{p+q+2}\left(\frac{1}{p+1}\int_{\Omega} H_u(u_j, v_j) u_j\, dx+ \frac{1}{q+1}\int_{\Omega}H_v(u_j, v_j) v_j\, dx-\int_{\Omega}H(u_j, v_j)\, dx\right)\\
&\quad + \frac{pq-1}{p+q+2}\int_{\Omega}H(u_j, v_j)dx\\
&\ge \frac{pq-1}{p+q+2}\int_{\Omega}H(u_j, v_j)dx -C.
\end{split}
\]
From $(H2)$ (see \cite[Lemma 1.1]{Felmer1993}), there exist $C_1, C_2>0$ such that
\[
H(u,v)\ge C_1(|u|^{p+1}+|v|^{q+1})-C_2, \quad \forall\; (u, v)\in \mathbf{E}_{\alpha}.
\]
Thus,
\begin{equation}\label{2210101503}
|u_j|_{p+1}^{p+1}+|v_j|_{q+1}^{q+1} \leq C+ o(1)\|(u_j, v_j)\|_{\mathbf{E}_{\alpha}}.
\end{equation}
On the other hand, it follows from \eqref{eq3.2} with $\psi=0$ that
\begin{equation}\label{2211151058}
 \int_{\Omega}A^{2s-\alpha}v_jA^{\alpha}\varphi dx = \int_{\Omega}H_u(u_j, v_j) \varphi  dx+o(1)\|\varphi\|_{E^\alpha},
\end{equation}
and if $\varphi=0$,
\begin{equation}\label{2211151618}
 \int_{\Omega}A^{\alpha}u_jA^{2s-\alpha}\psi dx  = \int_{\Omega}H_v(u_j, v_j) \psi dx+o(1)\|\psi\|_{E^{2s-\alpha}}.
\end{equation}
Applying $(H4)$, H\"older inequality and $E^{\alpha}\subset L^{p+1}(\Omega)$, there holds
\begin{equation}\label{2211151059}
\int_{\Omega}\left|H_u(u_j, v_j) \varphi \right| dx\le C\left( |u_j|_{p+1}^p+|v_j|_{q+1}^{p(q+1)/(p+1)} + 1\right)\|\varphi\|_{E^\alpha}.
\end{equation}
Since $A^{2s-\alpha}$ is an isomorphism from $E^{2s-\alpha}$ to $L^2(\Omega)$ and due to \eqref{2211151058}, \eqref{2211151059}, we have
\begin{equation}\label{2210101504}
\begin{aligned}
\|v_j\|_{E^{2s-\alpha}}& = |A^{2s-\alpha}v_j |_2=  \underset{ \|\varphi\|_{E^{\alpha}}=1}{\sup} \left| \int_{\Omega}A^{2s-\alpha}v_j A^{\alpha}\varphi dx  \right| \le C\left( |u_j|_{p+1}^p+|v_j|_{q+1}^{p(q+1)/(p+1)}+1\right).
\end{aligned}
\end{equation}
Similarly, we obtain
\begin{equation}\label{2211151112}
\begin{aligned}
\|u_j\|_{E^{\alpha}} \le C\left( |v_j|_{q+1}^q+|u_j|_{p+1}^{q(p+1)/(q+1)}+1\right).
\end{aligned}
\end{equation}
Combining \eqref{2210101503} with \eqref{2210101504} and \eqref{2211151112}, it is clear that $\{(u_j, v_j)\}$ is bounded in $\mathbf{E}_{\alpha}$. Therefore, up to a subsequence, there is $(u, v)\in \mathbf{E}_{\alpha}$ such that $(u_j, v_j) \rightharpoonup (u, v)$ weakly in $\mathbf{E}_{\alpha}$. Let
\[
w_j=H_u(u_j, v_j),\quad w=H_u(u, v),\quad y_j=H_v(u_j, v_j),\quad y=H_v(u, v).
\]
By Lemma \ref{2211151517}, up to a subsequence, $(w_j, y_j)\to (w, y)$ in $X$. In view of \eqref{2211151058} and \eqref{2211151618}, let $j$ tend to $\infty$, there holds
\begin{equation}\label{2211151817}
\begin{aligned}
 \int_{\Omega}A^{2s-\alpha}vA^{\alpha}\varphi dx = \int_{\Omega}w \varphi  dx,\quad \int_{\Omega}A^{\alpha}v A^{2s-\alpha}\psi dx  = \int_{\Omega}y \psi dx.
 \end{aligned}
 \end{equation}
Thanks to \eqref{2211151058}, \eqref{2211151817}, H\"older inequality and $E^{\alpha}\subset L^{p+1}(\Omega)$, we arrive at
 \[
 \begin{aligned}
 \|v_j-v\|_{E^{2s-\alpha}}&=\underset{\|\varphi\|_{E^{\alpha}}=1}{\sup}\int_{\Omega}A^{2s-\alpha}(v_j-v)A^{\alpha}\varphi dx \\
 &= \underset{\|\varphi\|_{E^{\alpha}}=1}{\sup}\left(\int_{\Omega}(w_j-w) \varphi  dx+o(1)\|\varphi\|_{E^\alpha}\right)\\
 &\le C\left(|w_j-w|_{1+\frac1p}+o(1)\right)
 \end{aligned}
 \]
which implies $v_j\to v$ in $E^{2s-\alpha}$. In the same way, we obtain $u_j\to u$ in $E^{\alpha}$.
\end{proof}

In the sequel, we shall verify that $\mathcal{E}$ possesses the linking structure stated in Theorem \ref{2210091640}.
\begin{lemma}\label{2211161004}
Let $p$, $q$ satisfy \eqref{condition}-\eqref{Condition_pq}, $H$ satisfy $(H1)$-$(H4)$. Then there exist two linear, bounded and invertible operators $B_1$, $B_2: \mathbf{E}_{\alpha}\to \mathbf{E}_{\alpha}$ such that for all $\tau\ge 0$, $\widehat{B}_{\tau}=P_2 B_1^{-1}e^{\tau L}B_2 : \mathbf{E}_{\alpha}^-\to \mathbf{E}_{\alpha}^-$ is invertible, where $P_2$ is the projection of $\mathbf{E}_{\alpha}$ onto $\mathbf{E}_{\alpha}^-$. Moreover, let 
$$e_1=\Big(\lambda_1^{-\frac{\alpha}{2s}}\varphi_1, \lambda_1^{\frac{-2s+\alpha}{2s}}\varphi_1\Big)\in \mathbf{E}_{\alpha}^+,$$ 
then there exist constants $R_2>\rho>0$, $R_1>\rho/\|B_1^{-1}B_2 e_1\|_{\mathbf{E}_{\alpha}}$ and $\sigma>0$ such that
$$
(G1)\;\; \mathcal{E}(z)\ge \sigma>0  \;\mbox{ on } S_\rho^+; \quad (G2)\;\; \mathcal{E}(z)\le 0 \;\mbox{ on } \partial Q$$
where $S_\rho^+:=\{B_1z^+: z^+\in \mathbf{E}^+_{\alpha}, \|z^+\|=\rho\}$
and
\[
Q:=\big\{B_2(te_1+z^-): 0\le t\le R_1, \, z^-\in \mathbf{E}_{\alpha}^-,\,  \|z^-\|_{\mathbf{E}_{\alpha}}\le R_2\big\}.
\]
\end{lemma}
\begin{proof}
Since $\frac1{p+1}+\frac{1}{q+1}<1$, we can select $\mu, \nu\ge 1$ satisfying
\begin{equation}\label{2302211611}
\frac{1}{p+1}<\frac{\mu}{\mu+\nu},\quad \frac{1}{q+1}<\frac{\nu}{\mu+\nu}.
\end{equation}
Let $R_1>1$, $0<\rho<1$ and define
\[
B_1(u, v)=(\rho^{\mu-1}u, \rho^{\nu-1}v), \;\; B_2(u, v)=(R_1^{\mu-1}u, R_1^{\nu-1}v),\quad \forall\; (u, v)\in \mathbf{E}_{\alpha},
\]
Since $L|_{ \mathbf{E}_{\alpha}^-}=-id_{\mathbf{E}_{\alpha}^-}$, $\widehat{B}_\tau$ is invertible in $\mathbf{E}_{\alpha}^-$. According to the definitions of $B_1$ and $B_2$, we set
\[
S_\rho^+ =\big\{(\rho^{\mu-1}u^+, \rho^{\nu-1}v^+): \|(u^+, v^+)\|_{\mathbf{E}_{\alpha}}=\rho,\, z^+=(u^+, v^+)\in \mathbf{E}^+_{\alpha}\big\},
\]
and
\[
\begin{aligned}
Q=\Big\{t(&R_1^{\mu-1}\widetilde{\varphi}_1, R_1^{\nu-1}\lambda_1^{\frac{-s+\alpha}{s}}\widetilde{\varphi}_1)+(R_1^{\mu-1}u^-, R_1^{\nu-1}v^-): 0\le t\le R_1, \\
&\|(u^-, v^-)\|_{\mathbf{E}_{\alpha}}\le R_2,\, z^-=(u^-, v^-)\in \mathbf{E}^-_{\alpha}\Big\},
\end{aligned}
\]
where $\widetilde{\varphi}_1=\lambda_1^{-\frac{\alpha}{2s}}\varphi_1$.

\smallskip
\noindent
{\bf Verification of $(G1)$}. For any $(\rho^{\mu-1}u^+, \rho^{\nu-1}v^+)\in S_\rho^+$, by $(H3)$ and $\mathbf{E}_{\alpha}\subset L^{p+1}(\Omega)\times L^{q+1}(\Omega)$, when $\rho$ is small, by \eqref{2302211611}, there exist $C_1, C_2 >0$ such that
\[
\begin{aligned}
\mathcal{E}(\rho^{\mu-1}u^+, \rho^{\nu-1}v^+)\ge \frac12 &\rho^{\mu+\nu-2}\|z^+\|^2_{\mathbf{E}_{\alpha}}-C_1\rho^{(\mu-1)(p+1)}\int_{\Omega}|u^+|^{p+1} dx -C_1\rho^{(\nu-1)(q+1)} \int_{\Omega}|v^+|^{q+1} dx\\
\ge \frac12 &\rho^{\mu+\nu-2}\|z^+\|^2_{\mathbf{E}_{\alpha}}-C_2\rho^{(\mu-1)(p+1)}\|z^+\|^{p+1}_{\mathbf{E}_{\alpha}}-C_2\rho^{(\nu-1)(q+1)} \|z^+\|_{\mathbf{E}_{\alpha}}^{q+1}\\
\ge \frac12 &\rho^{\mu+\nu}-C_2\rho^{\mu (p+1)}-C_2 \rho^{\nu (q+1)},
\end{aligned}
\]
which yields that there exist $\sigma,\, \rho>0$ satisfying $\mathcal{E}(\rho^{\mu-1}u^+, \rho^{\nu-1}v^+)\ge \sigma > 0$.

\vspace{0.5em}
\noindent
{\bf Verification of $(G2)$}. We proceed by the following three steps.

\smallskip
{\it Step 1}. For any $(R_1^{\mu-1}u^-, R_1^{\nu-1}v^-)\in Q\cap \{t=0\}$, it follows from $(H1)$ that
\[
\begin{aligned}
\mathcal{E}(R_1^{\mu-1}u^-, R_1^{\nu-1}v^-)\le R_1^{\mu+\nu-2}\int_{\Omega}A^{\alpha}u^- A^{2s-\alpha}v^-\, dx \le -R_1^{\mu+\nu-2}\int_{\Omega}|A^{\alpha}u^-|^2 dx \le 0.
\end{aligned}
\]

{\it Step 2}. For any $\widetilde{z}_{R_1}=(R_1^{\mu}\widetilde{\varphi}_1 +R_1^{\mu-1}u^-,\, R_1^{\nu}\lambda_1^{\frac{-s+\alpha}{s}}\widetilde{\varphi}_1 +R_1^{\nu-1}v^-)\in Q\cap \{t=R_1\}$, we write $u^-=r\varphi_1+w$ where $w\in E^{\alpha}$ is orthogonal to $\varphi_1$ in $L^2(\Omega)$.
\smallskip

Let $r\ge 0$. By direct computations,
\begin{equation}\label{2211152322}
r+\lambda_1^{-\frac{\alpha}{2s}}t=\int_{\Omega}(t\widetilde{\varphi}_1+ u^-){\varphi}_1 dx\le |t\widetilde{\varphi}_1+ u^-|_{p+1}|\varphi_1|_\frac{p+1}{p}\le C|t\widetilde{\varphi}_1+ u^-|_{p+1}.
\end{equation}
Set $\widetilde{z}_t=t(R_1^{\mu-1}\widetilde{\varphi}_1, R_1^{\nu-1}\lambda_1^{\frac{-s+\alpha}{s}}\widetilde{\varphi}_1)+(R_1^{\mu-1}u^-, R_1^{\nu-1}v^-)$ with $t\ge 0$.
Using $(H2)$ (see \cite[Lemma 1.1]{Felmer1993}) and \eqref{2211152322}, we have
\begin{equation}\label{2210112040}
\begin{aligned}
\mathcal{H}(\widetilde{z}_t)& \ge C_1R_1^{(p+1)(\mu-1)}\int_{\Omega}|t\widetilde{\varphi}_1+u^-|^{p+1} dx  + C_1R_1^{(q+1)(\nu-1)}\int_{\Omega}|t\lambda_1^{\frac{-s+\alpha}{s}}\widetilde{\varphi}_1 +v^-|^{q+1} dx-C_2\\
& \ge C_3 R_1^{(p+1)(\mu-1)}(r+\lambda_1^{-\frac{\alpha}{2s}}t)^{p+1}-C_2\\
 & \ge C_4 R_1^{(p+1)(\mu-1)}t^{p+1}-C_2.
\end{aligned}
\end{equation}
\par
If $r<0$, let $v^-=-A^{-2s+\alpha}\circ A^{\alpha}u^- = -A^{-2s+2\alpha} u^-$,
\[
\begin{aligned}
\int_{\Omega} v^- \varphi_1dx = \int_{\Omega}(-A^{-2s+2\alpha}u^-) \varphi_1 dx
& = \int_{\Omega}-A^{-2s+2\alpha} (r\varphi_1+w) \varphi_1 dx\\
&= -\lambda_1^{\frac{-s+\alpha}{s}}r-\int_{\Omega}\varphi_1A^{-2s+2\alpha} w dx\\
&= -\lambda_1^{\frac{-s+\alpha}{s}}r.
\end{aligned}
\]
Thus,
\[
\lambda_1^{\frac{-s+\alpha}{s}}(\lambda_1^{-\frac{\alpha}{2s}}t-r)=\int_{\Omega}(t\lambda_1^{\frac{-s+\alpha}{s}}\widetilde{\varphi}_1 +v^-)\varphi_1\, dx\le |t\lambda_1^{\frac{-s+\alpha}{s}}\widetilde{\varphi}_1 +v^-|_{q+1}|\varphi_1|_\frac{q+1}{q}.
\]
Hence, similar to \eqref{2210112040}, one has
\begin{equation}\label{2210112041}
\mathcal{H}(\widetilde{z}_t)\ge C_5 R_1^{(q+1)(\nu-1)}t^{q+1}-C_6.
\end{equation}
Therefore, either \eqref{2210112040} or \eqref{2210112041} holds, which yields
\[
\mathcal{E}(\widetilde{z}_{R_1})\le {R_1^{\mu+\nu}}- \frac{R_1^{\mu+\nu-2}}{2}\|z^-\|_{\mathbf{E}_{\alpha}}^2-C_4 R_1^{(p+1)\mu}+C_2,
\]
or
\[
\mathcal{E}(\widetilde{z}_{R_1})\le {R_1^{\mu+\nu}}- \frac{R_1^{\mu+\nu-2}}{2}\|z^-\|_{\mathbf{E}_{\alpha}}^2-C_5 R_1^{(q+1)\nu}+C_6.
\]
In both cases, we can choose $R_2=2{R_1}$ large so that $\mathcal{E}(\widetilde{z}_{R_1})\le 0$.

\smallskip
{\it Step 3}. For any $\|(u^-, v^-)\|_{\mathbf{E}_{\alpha}}=R_2$, there holds
$$\widetilde{z}_t=t(R_1^{\mu-1}\widetilde{\varphi}_1, R_1^{\nu-1}\lambda_1^{\frac{-s+\alpha}{s}}\widetilde{\varphi}_1)+(R_1^{\mu-1}u^-, R_1^{\nu-1}v^-)\in Q\cap \{\|z^-\|_{\mathbf{E}_{\alpha}}=R_2\}.$$
Since either \eqref{2210112040} or \eqref{2210112041} holds, we get either
\[
\mathcal{E}(\widetilde{z}_t)\le {R_1^{\mu+\nu-2}t^2}- \frac{R_1^{\mu+\nu-2}}{2}R_2^2-C_4 R_1^{(p+1)(\mu-1)}t^{p+1}+C_2,
\]
or
\[
\mathcal{E}(\widetilde{z}_t)\le {R_1^{\mu+\nu-2}t^2}- \frac{R_1^{\mu+\nu-2}}{2}R_2^2-C_5 R_1^{(q+1)(\nu-1)}t^{q+1}+C_6.
\]
Choosing $R_2=2{R_1}$ large, we get $\mathcal{E}(\widetilde{z}_t)\le 0$ for $0\le t \le R_1$.
\end{proof}

\begin{lemma}\label{2407281607}
Under the assumption $(H4)$ with subcritical $(p, q)$, any weak solution $(u, v)\in X^*$ of \eqref{main} is a classical solution.
\end{lemma}
\begin{proof} Since $(u, v)\in X^*$ and $(H4)$, one has $H_z(u, v)\in X$, so
\begin{equation}\label{2407281353}
(u, v)\in\mathcal{W}^{2s, 1+\frac{1}{q}}(\Omega)\times \mathcal{W}^{2s, 1+\frac{1}{p}}(\Omega).
\end{equation}
Consider first $p, q>\frac{2s}{N-2s}$, Proposition \ref{2302102144} yields then that $u\in L^{\frac{N(q+1)}{Nq-2s(q+1)}}(\Omega)$ and $v\in L^{\frac{N(p+1)}{Np-2s(p+1)}}(\Omega)$. The assumption $(H4)$ deduces that $(u, v)\in \mathcal{W}^{2s, r_1}(\Omega)\times \mathcal{W}^{2s, t_1}(\Omega)$ where
\[
r_1=\min\left\{\frac{N(p+1)}{[Nq-2s(p+1)]q}, \frac{N(q+1)^2}{[Nq-2s(q+1)](p+1)q}\right\},
\]
and
\[
t_1=\min\left\{\frac{N(q+1)}{[Nq-2s(q+1)]p}, \frac{N(p+1)^2}{[Np-2s(p+1)](q+1)p}\right\}.
\]
Since $(p, q)$ is subcritical, we have
\begin{equation}\label{2407281559}
r_1>r_0:=1+\frac{1}{q}, \quad t_1>t_0:=1+\frac{1}{p}.
\end{equation}
In the same way, whenever $r_n, t_n < \frac{N}{2s}$, let
\[
r_{n+1}=\min\left\{\frac{Nt_n}{(N-2st_n)q}, \frac{Nr_n(q+1)}{(N-2sr_n)(p+1)q}\right\}
\]
and
\[
t_{n+1}=\min\left\{\frac{Nr_n}{(N-2sr_n)p}, \frac{Nt_n(p+1)}{(N-2st_n)(q+1)p}\right\},
\]
then $(u, v)\in \mathcal{W}^{2s, r_{n+1}}(\Omega)\times \mathcal{W}^{2s, t_{n+1}}(\Omega)$, and $r_n < r_{n+1}, t_n < t_{n+1}$ due to \eqref{2407281559}. We claim that  
\begin{align}
\label{newclaim1}
\mbox{there exists $n$ such that $(u, v)\in \mathcal{W}^{2s, r_n}(\Omega)\times \mathcal{W}^{2s, t_n}(\Omega)$ with $\min(r_n, t_n)>\frac{N}{2s}$.}
\end{align}
This claim can be proved by contradiction. Assume for example $r_n\to r \le \frac{N}{2s}$, $t_n\to t \le \frac{N}{2s}$, other cases can be ruled out similarly. We can check that
$$r=\frac{N}{2s}\frac{pq-1}{pq+q}, \quad t=\frac{N}{2s}\frac{pq-1}{pq+p},$$
which contradicts $r>r_0$ since $(p, q)$ is subcritical. 

Therefore, after finite number of iterations, \eqref{newclaim1} holds true. Then $u, v\in C^s(\overline{\Omega})$, thanks again to Proposition \ref{2302102144} and Lemma \ref{2402061816}.

 Other situations with $p$ or $q\le \frac{2s}{N-2s}$ can be handled easily using Proposition \ref{2302102144} and Lemma \ref{2402061816}, so we omit the details.
\end{proof}

\begin{proof}[\bf Proof of Theorem \ref{220516} completed] By Lemmas \ref{2211161003}-\ref{2211161004}, Remarks \ref{2302211357}, \ref{2302211335}, applying Theorem \ref{2210091640} and Proposition \ref{2302302237}, there exists a $L^1$-weak solution $(u, v)\in X^*$. By Lemma \ref{2407281607}, $(u, v)$ is a classical solution, so the proof is completed.
\end{proof}

\section{Dual method and proof of Theorem \ref{2211092215}}\label{2306141749}
This section is devoted to the proof of Theorem \ref{2211092215}. We state first some basic properties for $H$, $H^*$ and $\mathcal{H}$, $\mathcal{H}^*$. As
\begin{equation}\label{2211052021}
H(u, v)= \int_0^1 H_u(tu, tv)u  + H_v(tu, tv)v dt, 
\end{equation}
there holds readily
\begin{lemma}\label{2211060953}
For $H$ satisfying $(H5)$, there are positive constants $A_1$, $A_2$ such that
\[
A_1 \left(|u|^{p+1}+ |v|^{q+1}\right)\le {H}(u, v)\le A_2\big(|u|^{p+1}+ |v|^{q+1}\big), \quad \forall\; (u, v) \in \R^2.
\]
\end{lemma}

Let $H^*$ be the Legendre-Fenchel transform of convex functional $H$, and $\mathcal{H}^*$ be the Legendre-Fenchel transform of $\mathcal{H}: X^*\to \mathbb{R}$, i.e.
\[
H^*(u, v)=\underset{(t, s)\in \mathbb{R}^2}{\sup}\left\{ tu+sv-H(t, s)\right\},\quad \forall\; (u,v)\in \mathbb{R}^2,
\]
and
\begin{equation}\label{2211072139}
\mathcal{H}^*(f, g)=\underset{(u, v)\in X^*}{\sup} \left\{ \int_{\Omega} (fu+gv) dx-\mathcal{H}(u, v)\right\},\quad \forall\; (f, g)\in X.
\end{equation}
Using Lemma \ref{2211060953} and Lemma \ref{2302252143} (d)(e), recalling that the Legendre-Fenchel transform reverses the order by Lemma \ref{2302252143} (a), we have
\begin{lemma}\label{2211011024}
Assume that $H$ is convex satisfying $(H5)$, then there are positive constants $A_3$, $A_4$ such that
\[
A_3\left(|u|^{1+\frac1p}+|v|^{1+\frac1q}\right)\le {H}^*(u, v)\le A_4\left(|u|^{1+\frac1p}+|v|^{1+\frac1q}\right), \quad \forall\; (u,v)\in \R^2
\]
and 
\[
A_3 \Big(|f|_{1+\frac1p}^{1+\frac1p}+ |g|_{1+\frac1q}^{1+\frac1q}\Big)\le \mathcal{H}^*(f, g)\le A_4 \Big(|f|_{1+\frac1p}^{1+\frac1p}+ |g|_{1+\frac1q}^{1+\frac1q}\Big), \quad \forall\; (f, g) \in X.
\]
\end{lemma}

Consider $\nabla H$ and  $\nabla H^*$ as mappings of $\R^2$. If $\nabla H$ is invertible, we denote by $(\nabla H)^{-1}$ its inverse.

\begin{lemma}\label{22112307}
Assume that $H$ satisfies $(H5)$-$(H6)$, then $H^*\in C^1(\mathbb{R}^2, \mathbb{R})$. Moreover $\nabla H : \mathbb{R}^2\to \mathbb{R}^2$ is a homeomorphism, and $(\nabla H)^{-1} = \nabla H^*$.
\end{lemma}
\begin{proof}
By Lemma \ref{2211060953}, there holds $\lim_{|z|\to \infty}{\frac{H(z)}{|z|}}=\infty$. Applying Lemma \ref{2302252143} (c), $H^*\in C^1(\mathbb{R}^2)$. 

As $\nabla H: \mathbb{R}^2\to \mathbb{R}^2$ is strictly monotone, $\nabla H$ is injective. We claim that $\nabla H$ is surjective. This fact should be known, we give a proof here for the sake of completeness. Indeed, for any $R_1>0$, by $(H5)$, there exists $R_2>0$ such that if $|z|\ge R_2$, then $|\nabla H(z)|>R_1$. Obviously, for any $w \in B_{R_1}$, $\text{deg} (\nabla H, B_{R_2}, w)$ is well defined and
\begin{equation}\label{2401132316}
\text{deg} (\nabla H, B_{R_2}, w)=\text{deg} (\nabla H, B_{R_2}, 0),
\end{equation}
where ``$\text{deg}$" denotes topological degree, see \cite[Appendix]{Willem96}.
We define a homotopy $F: B_{R_2}\times [0,1]\to \mathbb{R}^2$,
\begin{equation}\label{2211051915}
F(z, t)=\nabla H\left(\frac{z}{1+t}\right)-\nabla H\left(\frac{-tz}{1+t}\right).
\end{equation}
For any $z\in \partial B_{R_2}$ and $t\in [0, 1]$, we have $\frac{z}{1+t}\neq \frac{-tz}{1+t}$, hence $F(z, t)\neq 0$ as $\nabla H$ is injective. Therefore,
\[
\text{deg} (F(\cdot, 1), B_{R_2}, 0)=\text{deg} (F(\cdot, 0), B_{R_2}, 0)=\text{deg} (\nabla H, B_{R_2}, 0).
\]
According to Borsuk theorem \cite[Theorem D.17]{Willem96}, $\text{deg} (F(\cdot, 1), B_{R_2}, 0)$ is odd, which implies
$$\text{deg} (\nabla H, B_{R_2}, 0)\neq 0.$$
So by \eqref{2401132316} we get
$$\text{deg} (\nabla H, B_{R_2}, w)\neq 0,\quad \forall \; w\in B_{R_1},$$
which deduces that there exists some $z\in B_{R_2}$ such that $\nabla H(z)=w$. Due to the arbitrariness of $R_1$, we complete the proof of this claim.

By Brouwer's invariance of domain theorem, it's known that any continuous bijection of Euclidean space is indeed a homeomorphism; we can get this fact also by noting that $\nabla H$ is proper, i.e.~the preimage of any bounded set is bounded seeing $(H5)$, so $\nabla H$ is a closed mapping, hence a homeomorphism. Finally, it follows from Lemma \ref{2302252143} (b) that $(\nabla H)^{-1}=\nabla H^*$.
\end{proof}

Now we consider the functional version of the previous Lemma.  Recall that $H_z(u, v) = \nabla H(u, v)$ and $H_z^*(u, v) = \nabla H^*(u, v)$.
\begin{lemma}\label{2302260102}
$H_z$ is a homeomorphism from $X^*$ onto $X$, and $H_z^{-1}= H_z^*$.
\end{lemma}
\begin{proof}
We claim that $H_z$ is continuous. Consider $(u_n, v_n)\to (u, v)$ in $X^*$, as before, up to a subsequence, there exists $(\Phi_1, \Phi_2)\in X^*$ such that $|u_n|\le \Phi_1$, $|v_n|\le \Phi_2$ a.e.~for all $n$; $u_n \to u$ and $v_n \to v$ a.e.~in $\Omega$. Using $(H5)$ once more, we obtain
\[
|H_u(u_n, v_n)|\le C\Big(|\Phi_1|^p+ |\Phi_1|^{\alpha-1}|\Phi_2|^{\beta}\Big)\in L^{1+\frac1p}(\Omega)
\]
and 
\[
|H_v(u_n, v_n)| \le C\Big(|\Phi_2|^q+ |\Phi_1|^{\alpha}|\Phi_2|^{\beta-1}\Big) \in L^{1+\frac1q}(\Omega).
\]
By Lebesgue's theorem, we deduce that $H_z(u_n, v_n) \to H_z(u, v)$ in $X$, which means $H_z$ is continuous.

Given any $(f, g)\in X$, $(u, v)= (\nabla H)^{-1}(f, g) = \nabla H^*(f, g) \in X^*$ thanks to $(H5)$, since $|u(x)|^p \leq C|f(x)|$ and $|v(x)|^q \leq C|g(x)|$. The bijectivity of $H_z$ is an easy consequence by definition. The continuity of $H_z^{-1} = H^*_z$ can be proved as for $H_z$, we omit the details. 
\end{proof}

Finally, we state properties of the Legendre-Fenchel transform $\mathcal{H}^*$ for the Hamiltonian functional $\mathcal{H}$.
\begin{lemma}\label{2211072137}
$\mathcal{H}^{*}\in C^1(X, \mathbb{R})$. More precisely,
\begin{equation}\label{2211052325}
\mathcal{H}^{*}(f, g)=\int_{\Omega}H^*(f, g) dx, \quad 
\left\langle (\mathcal{H}^{*})'(f, g), (\widetilde{f}, \widetilde{g}) \right\rangle=\int_{\Omega}H^*_u(f, g) \widetilde{f}+H^*_v(f, g)\widetilde{g} dx.
\end{equation}
\end{lemma}

\begin{proof}
For any $(f, g)\in X$, by Lemma \ref{2302252143} (b) and Lemma \ref{22112307}, there exist functions $u, v$ such that
\[
H^*(f, g)=fu + gv -H(u, v),
\]
and $(u, v) = H_z^*(f, g) = \nabla H^*(f, g)$. By Lemma \ref{2302260102}, we have $(u, v)\in X^*$. In view of \eqref{2211072139}, it follows that
\[
\begin{aligned}
\mathcal{H}^*(f, g) = \underset{(\widetilde{u}, \widetilde{v})\in X^*}{\sup} \int_{\Omega}f\widetilde{u}+g\widetilde{v}-H(\widetilde{u}, \widetilde{v}) dx & \le \int_{\Omega}\underset{(t, s)\in \mathbb{R}^2}{\sup} \left\{f(x)t+g(x)s-H(t, s) \right\} dx\\
& = \int_{\Omega}H^*(f, g) dx\\
& = \int_{\Omega} fu+gv-H(u,v) dx\\
& \le \mathcal{H}^*(f, g),
\end{aligned}
\]
which implies the expression of $\mathcal{H}^*$. By $H^*\in C^1(\mathbb{R}^2, \mathbb{R})$ and Lemma \ref{2302260102}, we get the expression of $(\mathcal{H}^{*})'$ and $\mathcal{H}^{*}\in C^1(X^*, \mathbb{R})$.
\end{proof}
Next, we explain how to find weak solutions to \eqref{main}. Clearly, $(-\Delta)^s$ is an isomorphism of $\mathcal{W}^{2s, r}(\Omega)$ onto $L^r(\Omega)$. Therefore we can denote its inverse by $\mathcal{A}: L^r(\Omega)\to \mathcal{W}^{2s, r}(\Omega)$. Define $\mathcal{J}: X\to \mathbb{R}$ by
\begin{equation}\label{2211011351}
\mathcal{J}(f, g)=\mathcal{H}^*(f, g)-\int_{\Omega}g\mathcal{A}f dx.
\end{equation}
Using Proposition \ref{2302102144}, $\mathcal{J}$ is well defined. 
\begin{lemma}\label{2211211644}
$\mathcal{A}$ is self-adjoint in the following sense: 
\begin{equation}\label{2211211634}
\int_{\Omega}g\mathcal{A}f dx=\int_{\Omega}f\mathcal{A}g dx, \quad \forall\; (f, g)\in X.
\end{equation}
\end{lemma}
\begin{proof}
In fact, this is a direct consequence of the self-adjointness of $(-\Delta)^s$ over $X_0^s(\Omega)$ and density argument. Let $\{f_n\}, \{g_n\}\subset C_c^\infty (\Omega)$ satisfying $(f_n, g_n)\to (f, g)$ in $X$. Set $(u_n, v_n) = (\mathcal{A}f_n, \mathcal{A}g_n)\in X_0^{s}(\Omega)^2$. Then
\[
\begin{split}
\int_{\Omega}g_n\mathcal{A}f_n dx=\int_{\Omega} u_n (-\Delta)^sv_n dx=\int_{\Omega}v_n (-\Delta)^s u_n dx=\int_{\Omega}f_n\mathcal{A}g_n dx.
\end{split}
\]
Taking $n\to \infty$, we get \eqref{2211211634}.
\end{proof}
We are now in position to present the dual method setting.
\begin{proposition}\label{2211171427}
If $(f, g)\in X$ is a critical point of $\mathcal{J}$, then ${H}_z^{*}(f, g)$ is a $L^1$-weak solution to \eqref{main}.
\end{proposition}
\begin{proof}
Since $(f, g)\in X$ is a critical point of $\mathcal{J}$,
\begin{equation}\label{2211211641}
0=\langle \mathcal{J}'(f, g), (f_1, g_1) \rangle=\langle (\mathcal{H}^*)'(f, g), (f_1 , g_1 )\rangle-\int_{\Omega}g\mathcal{A}f_1 dx-\int_{\Omega}g_1\mathcal{A}f dx,\quad \forall\; (f_1, g_1)\in X.
\end{equation}
Let $(u, v)=H_z^*(f, g)$, then by Lemma \ref{2211072137}, we have
\begin{equation}\label{2211211642}
\langle (\mathcal{H}^*)'(f, g), (f_1 , g_1 )\rangle=\int_{\Omega}uf_1 dx +\int_{\Omega}v g_1 dx.
\end{equation}
As $H_z^*=H_z^{-1}$, we obtain
$(f, g)=(H_u(u, v), H_v(u,v))$.
If we choose $g_1=0$, then using \eqref{2211211641}, \eqref{2211211642} and Lemma \ref{2211211644}, it holds that
\[
\int_{\Omega}uf_1 dx=\langle (\mathcal{H}^*)'(f, g), (f_1, 0) \rangle=\int_{\Omega}g\mathcal{A}f_1 dx=\int_{\Omega}f_1\mathcal{A}g dx, \quad \forall\; f_1 \in L^{1+\frac{1}{p}}(\Omega).
\]
Hence $u=\mathcal{A}g$, i.e.~$(-\Delta)^s u=g=H_v(u, v)$ in the weak sense. 
Similarly, $(-\Delta)^s v=H_u(u, v)$.
\end{proof}

\subsection{Proof of Theorem \ref{2211092215} completed}
Now we establish a mountain pass structure to get existence of nontrivial critical points of $\mathcal{J}$. Set
$$S_\rho:=\{(\rho^{k-1}f, \rho^{l-1}g): (f, g)\in X, \,\|(f, g)\|_X=\rho\}$$
where $k, l>1$ satisfy
\begin{equation}\label{2211012130}
\frac{p}{p+1}>\frac{k}{k+l}, \quad \frac{q}{q+1}>\frac{l}{k+l}.
\end{equation}
Assume $0<\rho<1$. For any $(\rho^{k-1}f, \rho^{l-1}g)\in S_\rho$, by Lemma \ref{2211011024} and Proposition \ref{2302102144},
\[
\begin{split}
\mathcal{J}(\rho^{k-1}f, \rho^{l-1}g)& =\mathcal{H}^*(\rho^{k-1}f, \rho^{l-1}g)-\rho^{k+l-2}\int_{\Omega}g\mathcal{A}f dx\\
& \ge C_1 \rho^{(k-1)(1+1/p)}|f|_{1+\frac1p}^{1+\frac1p}+C_1\rho^{(l-1)(1+1/q)}|g|_{1+\frac1q}^{1+\frac1q}-C_2\rho^{k+l}\\
& \ge C_3\rho^{\max\{k(1+1/p),\, l(1+1/q)\}}-C_2\rho^{k+l}.
\end{split}
\]
According to \eqref{2211012130}, there exist $\rho_0 > 0, \beta>0$ such that $\mathcal{J}(\rho_0^{k-1}f, \rho_0^{l-1}g)>\beta$ if $(\rho_0^{k-1}f, \rho_0^{l-1}g)\in S_{\rho_0}$. 
On the other hand, we fix some $(f_0, g_0)\in X$ with $$\int_{\Omega}f_0\mathcal{A}g_0 dx>0.$$ By Lemma \ref{2211011024},
\[
\mathcal{J}(\rho^{k}f_0, \rho^{l}g_0)\le C_4 \rho^{k(1+1/p)}|f_0|_{1+\frac1p}^{1+\frac1p}+C_4\rho^{l(1+1/q)} |g_0|_{1+\frac1q}^{1+\frac1q}-\rho^{k+l}\int_{\Omega}f_0\mathcal{A}g_0 dx.
\]
Consequently, due to \eqref{2211012130}, there exists $\rho_1>\rho_0$ such that $\mathcal{J}(\rho_1^{k}f_0, \rho_1^{l}g_0)<0$. Let
\[
\Gamma:=\{\gamma\in C([0, 1], X): \gamma(0)=0,\, \gamma(1)=(\rho_1^{k}f_0, \rho_1^{l}g_0) \}.
\]
It is clear that $\gamma([0, 1])\cap S_{\rho_0}\neq \varnothing$ for any $\gamma\in \Gamma$, from which we obtain a mountain pass structure of $\mathcal{J}$ around 0, and define the mountain pass level by
\[
c:=\underset{\gamma\in \Gamma}{\inf}\,\underset{t\in [0, 1]}{\max}\, \mathcal{J}(\gamma(t))\ge \beta>0.
\]

\begin{lemma}\label{2211171422}
$\mathcal J$ satisfies the Palais-Smale condition.
\end{lemma}
\begin{proof}
 Let $\{(f_n, g_n)\}\subset X$ be a Palais-Smale consequence of $\mathcal{J}$ at level $c$, that is,
\begin{equation}\label{2211041545}
\mathcal{J}(f_n, g_n)\to c\;\; \mbox{and }\; \mathcal{J}'(f_n, g_n)\to 0\;\; \text{in}~X^*, \quad \mbox{as $n\to \infty$}.
\end{equation}
From \eqref{2211041545} and \eqref{2211211634}, it follows that
\begin{equation}\label{2211051603}
\mathcal{H}^*(f_n, g_n)-\langle (\mathcal{H}^*)'(f_n, g_n), ((1-\theta) f_n, \theta g_n)\rangle=c+o(1)\|(f_n, g_n)\|_{X}+o(1)
\end{equation}
where $\theta$ is given in $(H7)$. Let $(u_n ,v_n)=H^*_z(f_n, g_n)$. According to Lemma \ref{2302252143} (b) and Lemma \ref{2211072137},
\begin{equation}\label{2302221345}
\mathcal{H}^*(f_n, g_n)=\int_{\Omega}f_nu_n+g_nv_ndx - \mathcal{H}(u_n, v_n),
\end{equation}
which together with \eqref{2211051603} and $(H7)$ implies that
\[
\begin{aligned}
&\quad c+o(1)\|(f_n, g_n)\|_{X}+o(1)\\
& = \int_{\Omega}f_nu_n+g_nv_n dx - \mathcal{H}(u_n, v_n)-\langle (\mathcal{H}^*)'(f_n, g_n), ((1-\theta) f_n, \theta g_n)\rangle\\
& = \theta\int_{\Omega} H_u(u_n, v_n)u_ndx+(1-\theta)\int_{\Omega} H_v(u_n, v_n)v_ndx- \mathcal{H}(u_n, v_n)\\
& \ge C_1\int_{\Omega}|u_n|^{p+1}dx+C_1\int_{\Omega}|v_n|^{q+1}dx-C_2|\Omega|.
\end{aligned}
\]
Combining the above inequality, \eqref{2302221345}, Lemma \ref{2211060953}, Lemma \ref{2211011024} and Young's inequality,  $\{(f_n, g_n)\}$ is bounded in $X$.

Furthermore, by Proposition \ref{2302102144}, $\mathcal{A}: X\to X^*$ is compact, hence $\{(\mathcal{A}f_n, \mathcal{A}g_n)\}$ is compact in $X^*$. By \eqref{2211041545},
\begin{equation}\label{2211050929}
\mathcal{J}'(f_n, g_n)=H^*_z(f_n, g_n)-(\mathcal{A}g_n, \mathcal{A}f_n)=o(1)\quad \text{in}~X^*.
\end{equation}
Since $H^*_z$ is a homeomorphism from $X$ onto $X^*$, $\{(f_n, g_n)\}$ is compact in $X$. The proof is done.
\end{proof}

As $\mathcal J$ satisfies the $(PS)$ condition, applying mountain pass theorem \cite{AR1973} and Proposition \ref{2211171427}, we get a nontrivial solution $(u, v)\in X^*$ of \eqref{main}. Since $(H5)$ implies $(H4)$, by Lemma \ref{2407281607}, $(u, v)$ is a classical solution.

\section{Positive solutions for fractional Lane-Emden system}\label{2306081126}
As very special case of \eqref{main}, we consider the system \eqref{2305191141} under the subcritical assumption \eqref{2305191926}. Using Lemma \ref{2407281607}, we know that any energy solution of \eqref{2305282030} belongs to $C^s(\overline{\Omega})$.

\begin{proposition}\label{2306130008}
Let $p,\, q$ satisfy \eqref{2305191926}. Then $u$ is an energy solution to \eqref{2305282030} if and only if $(u, v)$ is a classical solution to \eqref{2305191141}, where $v=|(-\Delta)^{s}u|^{\frac1q-1}(-\Delta)^{s}u$.
\end{proposition}
\begin{proof}
First, assume that $u$ is an energy solution to \eqref{2305282030}. It follows from \eqref{2305282038} that
\begin{equation}\label{2306012027}
\int_{\Omega}v(-\Delta)^{s}\varphi dx=\int_{\Omega}|u|^{p-1}u\varphi dx\quad \forall\; \varphi\in\mathcal{T}_s(\Omega).
\end{equation}
Next, since $v=|(-\Delta)^{s}u|^{\frac1q-1}(-\Delta)^{s}u$ and $u\in\mathcal{W}^{2s, 1+\frac1q}(\Omega)$, we have
\begin{equation}\label{2306012028}
\int_{\Omega}u(-\Delta)^{s}\varphi dx=\int_{\Omega}|v|^{q-1}v\varphi dx\quad \forall\; \varphi\in\mathcal{T}_s(\Omega).
\end{equation}
By \eqref{2306012027} and \eqref{2306012028}, $(u, v)$ is a weak solution to \eqref{2305191141}. Since $u$ is an energy solution to \eqref{2305282030}, using Lemma \ref{2407281607}, $u,\, v\in C^s(\overline{\Omega})$. 

Conversely, if $(u, v)$ is a classical solution to \eqref{2305191141}, \eqref{2306012027} and \eqref{2306012028} will hold, which implies
\begin{equation}\label{2306012036}
\int_{\Omega}|(-\Delta)^{s}u|^{\frac1q-1}(-\Delta)^{s}u(-\Delta)^{s}\varphi dx=\int_{\Omega}|u|^{p-1}u\varphi dx\quad \forall\; \varphi\in \mathcal{W}^{2s, 1+\frac1q}(\Omega).
\end{equation}
 Thus, $u$ is an energy solution to \eqref{2305282030}.
\end{proof}

We denote
\begin{equation}\label{2306031629}
\mathcal{R}(u):=\langle \mathcal{I}'(u), u\rangle=\left|(-\Delta)^su\right|_{1+\frac1q}^{1+\frac1q}-|u|_{p+1}^{p+1},\quad \forall\; u\in \mathcal{W}^{2s, 1+\frac1q}(\Omega).
\end{equation}
For any $u\in\mathcal{N}_{\mathcal{I}}$ (see \eqref{2406220952}), there holds
\[
\begin{aligned}
\langle \mathcal{R}'(u), u\rangle =\Big(1+\frac{1}{q}\Big)\left|(-\Delta)^su\right|_{1+\frac1q}^{1+\frac1q}-(p+1)|u|_{p+1}^{p+1} &=\Big(\frac{1}{q}-p\Big)\left|(-\Delta)^su\right|_{1+\frac1q}^{1+\frac1q},
\end{aligned}
\]
which implies
\begin{equation}\label{2306032031}
\begin{cases}
\langle \mathcal{R}'(u), u\rangle<0,\quad \text{if}~~pq>1,\\
\langle \mathcal{R}'(u), u\rangle>0,\quad \text{if}~~pq<1.
\end{cases}
\end{equation}
Using implicit function theorem, if $pq\neq 1$, $\mathcal{N}_{\mathcal{I}}$ is a $C^1$-submanifold of $\mathcal{W}^{2s, 1+\frac1q}(\Omega)$ with codimension 1.
\begin{lemma}\label{2306031952}
If $pq\neq 1$, then $\mathcal{N}_{\mathcal{I}}$ is non empty. Moreover, when $pq>1$, $\mathcal{N}_{\mathcal{I}}$ is far away from zero and $\mathcal{I}$ constrained on $\mathcal{N}_{\mathcal{I}}$ has a positive lower bound; when $pq<1$, $\mathcal{N}_{\mathcal{I}}$ is bounded in $\mathcal{W}^{2s, 1+\frac1q}(\Omega)$.
\end{lemma}
\begin{proof}
For any $u\in \mathcal{W}^{2s, 1+\frac1q}(\Omega)\backslash\{0\}$, when $pq>1$ (resp.~$pq<1$), let $t_u$ be the maximum (resp.~minimum) point of $\mathcal{I}(tu)$ for $t>0$. Then $t_uu\in \mathcal{N}_{\mathcal{I}},$
so $\mathcal{N}_{\mathcal{I}}$ is non empty. For any $u\in \mathcal{N}_{\mathcal{I}}$, using \eqref{2306031629} and $\langle \mathcal{I}'(u), u\rangle=0$, we have
\begin{equation}\label{2306122218}
\begin{aligned}
\left|(-\Delta)^su\right|_{1+\frac1q}^{1+\frac1q}=|u|_{p+1}^{p+1}\le C\left|(-\Delta)^su\right|_{1+\frac1q}^{p+1}.
\end{aligned}
\end{equation}
Therefore if $pq>1$, there exists some $C_0>0$ such that
\begin{equation}\label{2305191933}
\left|(-\Delta)^su\right|_{1+\frac1q}>C_0,\quad \forall\; u\in \mathcal{N}_{\mathcal{I}}.
\end{equation}
Next, if $u\in \mathcal{N}_{\mathcal{I}}$, by \eqref{2306031629},
\[
\begin{aligned}
\mathcal{I}(u)=\left(\frac{q}{q+1}-\frac1{p+1}\right)\left|(-\Delta)^su\right|_{1+\frac1q}^{1+\frac1q}.
\end{aligned}
\]
From \eqref{2305191933} and $pq>1$, it follows that $\mathcal{I}$ has a positive lower bound on $\mathcal{N}_{\mathcal{I}}$. When $pq<1$, we know from \eqref{2306122218} that $\mathcal{N}_{\mathcal{I}}$ is bounded in $\mathcal{W}^{2s, 1+\frac1q}(\Omega)$.
\end{proof}
In the sequel, we consider the functional $\mathcal{I}$ constrained on $\mathcal{N}_\mathcal{I}$.
A constrained critical point $u$ of $\mathcal{I}|_{\mathcal{N}_\mathcal{I}}$, means that there exists a Lagrange multiplier $\lambda\in \mathbb{R}$ such that
\begin{equation}\label{2306032016}
\mathcal{I}'(u)=\lambda \mathcal{R}'(u) \quad \text{in}~\mathcal{W}^{2s, 1+\frac1q}(\Omega)^*,
\end{equation}
where $\mathcal{W}^{2s, 1+\frac1q}(\Omega)^*$ is the dual space of $\mathcal{W}^{2s, 1+\frac1q}(\Omega)$. In particular, $u$ is a critical point of $\mathcal{I}$ whenever $\lambda=0$.
\begin{lemma}\label{2306122343}
Assume that $pq\neq 1$. Then any constrained critical point of $\mathcal{I}|_{\mathcal{N}_\mathcal{I}}$ is a critical point of $\mathcal{I}$.
\end{lemma}
\begin{proof}
Suppose that $u$ is a constrained critical point, there is $\lambda\in\mathbb{R}$ such that \eqref{2306032016} holds. Thus,
\begin{equation}\label{2305192023}
0=\langle \mathcal{I}'(u), u \rangle=\lambda  \langle \mathcal{R}'(u), u\rangle=\lambda \left(\frac{1}{q}-p\right)\left|(-\Delta)^su\right|_{1+\frac1q}^{1+\frac1q},
\end{equation}
which deduces $\lambda=0$. Consequently, $u$ is a critical point of $\mathcal{I}$
\end{proof}
\begin{lemma}\label{2306031956}
For $pq\neq 1$, $\mathcal{I}|_{\mathcal{N}_\mathcal{I}}$ satisfies the $(PS)$ condition.
\end{lemma}
\begin{proof}
Let $\{u_n\}$ be a $(PS)$ sequence at level $c\in \mathbb{R}$ for $\mathcal{I}|_{\mathcal{N}_\mathcal{I}}$, that is $\{u_n\} \subset \mathcal{N}_\mathcal{I}$, $\lambda_n \in \R$ such that
\begin{equation}\label{2305200030}
\mathcal{I}(u_n)=c+o(1) \quad \mbox{and} \quad \mathcal{I}'(u_n)-\lambda_n\mathcal{R}'(u_n)=o(1)\;\; \text{in}~\mathcal{W}^{2s, 1+\frac1q}(\Omega)^*.
\end{equation}
Therefore
\[
\begin{aligned}
c+o(1) =\mathcal{I}(u_n)-\frac{1}{p+1}\mathcal{R}(u_n) &=\left(\frac{q}{q+1}-\frac{1}{p+1}\right)\left|(-\Delta)^su_n\right|_{1+\frac1q}^{1+\frac1q}.
\end{aligned}
\]
By $pq\neq 1$, one can conclude the boundedness of $\{u_n\}$ in $\mathcal{W}^{2s, 1+\frac1q}(\Omega)$, which yields also that $\{\langle \mathcal{R}'(u_n), u_n\rangle\}$ is bounded. Up to a subsequence, assume that
\begin{equation}\label{2306032345}
\langle \mathcal{R}'(u_n), u_n\rangle\to m \in \R.
\end{equation}
Let $m\neq 0$, \eqref{2305200030} yields $o(1)=\langle\mathcal{I}'(u_n)-\lambda_n\mathcal{R}'(u_n), u_n\rangle=\lambda_nm+o(1),$
which yields $\lambda_n\to 0$. Up to a new subsequence, there exists $u\in \mathcal{W}^{2s, 1+\frac1q}(\Omega)$ such that $(-\Delta)^su_n\rightharpoonup (-\Delta)^su$ in $L^{\frac{q+1}{q}}(\Omega)$,
and $u_n\to u$ in $L^{p+1}(\Omega)$, so $u\in \mathcal{N}_{\mathcal{I}}$. By direct computations,
\[
o(1)=|u_n|_{p+1}^{p+1}-|u|_{p+1}^{p+1}=\left|(-\Delta)^su_n\right|_{1+\frac1q}^{1+\frac1q}-\left|(-\Delta)^su\right|_{1+\frac1q}^{1+\frac1q},
\]
which together with the weak convergence of $(-\Delta)^s u_n$ in $L^{1 +\frac{1}{q}}(\Omega)$ yields that $u_n\to u$ in $\mathcal{W}^{2s, 1+\frac1q}(\Omega)$.\par

Let now $m=0$. It is clear from \eqref{2306032345} and $u_n\in \mathcal{N}_{\mathcal{I}}$ that
\[
\begin{aligned}
o(1)=\langle \mathcal{R}'(u_n), u_n\rangle=\Big(\frac{1}{q}-p\Big)\left|(-\Delta)^su_n\right|_{1+\frac1q}^{1+\frac1q},
\end{aligned}
\]
which deduces that $\left|(-\Delta)^su_n\right|_{1+\frac1q}\to 0.$ So we are done.
\end{proof}

\begin{proof}[\bf Proof of Theorem \ref{2306122224} completed]
(i) We apply a deformation lemma on $C^1$ manifold (see \cite{Bonnet1993}) to $\mathcal{I}|_{\mathcal{N}_\mathcal{I}}$. By Lemma \ref{2306031952}, it is easy to see that $c_{\mathcal{I}}<0$ if $pq<1$ and $c_{\mathcal{I}}>0$ if $pq<1$. By Lemma \ref{2306031956}, $c_{\mathcal{I}}$ is attained by some $u\in \mathcal{N}_\mathcal{I}$. Hence, $u$ is a constrained critical point of $\mathcal{I}|_{\mathcal{N}_\mathcal{I}}$. Using Lemma \ref{2306122343}, $u$ is also a critical point of $\mathcal{I}$, hence an energy solution of \eqref{2305282030}. Let $w$ be the weak solution of
\[
(-\Delta)^{s} w =|(-\Delta)^{s}u| \; \text{ in }\;\Omega,\quad u=0~\;\text{ in} \;\;\mathbb{R}^N\setminus\Omega.
\]
By Lemma \ref{2306081933}, $w\ge \pm u$, so $w\ge |u|$. Consequently, $\mathcal{I}(tw)\le \mathcal{I}(tu)$ for all $t>0$. Thus there exists a unique $t_w>0$ such that $t_ww\in \mathcal{N}_\mathcal{I}$ and $\mathcal{I}(t_ww)\le  \mathcal{I}(u)=c_{\mathcal{I}}.$ Then $t_ww$ is a minimizer for $c_{\mathcal{I}}$. By means of Lemma \ref{2306012122}, one has $t_ww>0$. The proof can be concluded by Proposition \ref{2306130008}.

(ii) Assume that $u_1, u_2\in \mathcal{W}^{2s, 1+\frac1q}(\Omega)$ are two distinct positive energy solutions of \eqref{2305282030}. We know that $u_1, u_2\in C^s(\overline{\Omega})$. Denote
\[
\beta_0:=\sup\{l\in\mathbb{R}: u_2\ge lu_1 ~\text{a.e. in} ~\Omega \} \quad\text{and}\quad
\beta_1:=\sup\{l\in\mathbb{R}: u_1\ge lu_2 ~\text{a.e. in} ~\Omega \}.
\]
From Lemma \ref{2306012122} and Lemma \ref{2402061816}, there exist $C_1, C_2>0$ such that $C_1\delta^s\le u_1, u_2\le C_2\delta^s$. Therefore $0<\beta_0$, $\beta_1<\infty$. In order to prove $u_1=u_2$, it suffices to prove that $\beta_0 \ge1$ and $\beta_1\ge 1$. Notice that
\[
\begin{aligned}
(-\Delta)^{s}\left(|(-\Delta)^{s}u_1|^{\frac1q-1}(-\Delta)^{s}u_1\right)&=u_1^{p}\ge  \beta_1^{p} u_2^{p}\\
&=\beta_1^{p}(-\Delta)^{s}\left(|(-\Delta)^{s}u_2|^{\frac1q-1}(-\Delta)^{s}u_2\right)\\
&=(-\Delta)^{s}\left(|(-\Delta)^{s}(\beta_1^{pq}u_2)|^{\frac1q-1}(-\Delta)^{s}(\beta_1^{pq}u_2)\right).\\
\end{aligned}
\]
By comparison principle, we get $u_1\ge \beta_1^{pq}u_2$. So $\beta_1\ge\beta_1^{pq}$, hence $\beta_1\ge 1$. A similar argument gives $\beta_0\ge 1$. By Proposition \ref{2306130008}, the classical solution of \eqref{2305191141} is unique.

(iii) Replacing $\mathcal{W}^{2s, 1+\frac1q}(B_R)$ by $\mathcal{W}_{{\rm rad}}^{2s, 1+\frac1q}(B_R)$, we can check that all the above steps work, so we have a positive classical radially symmetric solution.
\end{proof}

\end{document}